\documentclass[a4paper,12pt]{amsart}

\usepackage{amsmath,amstext,amssymb,amsopn,amsthm}
\usepackage{url}
\usepackage{mathtools}

\usepackage[margin=30mm]{geometry}
\usepackage{eucal,mathrsfs,dsfont}

\newtheorem{theorem}{Theorem}[section]
\newtheorem{corollary}[theorem]{Corollary}
\newtheorem{lemma}[theorem]{Lemma}
\newtheorem{proposition}[theorem]{Proposition}

\theoremstyle{definition}

\theoremstyle{remark}
\newtheorem{remark}[theorem]{Remark}

\numberwithin{equation}{section}

\newcommand{\eps}{\varepsilon}

\newcommand{\calA}{\mathcal{A}}

\newcommand{\calG}{\mathcal{G}}
\newcommand{\calD}{\mathcal{D}}

\newcommand{\calP}{\mathcal{P}}

\newcommand{\I}{\mathds{1}}
\newcommand{\Ee}{\operatorname{\mathds{E}}} 
\newcommand{\Pp}{\operatorname{\mathds{P}}} 
\newcommand{\R}{\mathds{R}}

\newcommand{\N}{{\mathds{N}}}
\newcommand{\C}{{\mathds{C}}}
\newcommand{\Z}{{\mathds{Z}}}

\newcommand{\Rt}{{\R^2}}
\newcommand{\uni}{\mathds{S}}

\newcommand{\ol}{\overline}
\newcommand{\wh}{\widehat}
\newcommand{\wt}{\widetilde}
\newcommand{\unif}{\mathcal{U}}

\DeclareMathOperator{\Arg}{Arg}
\DeclareMathOperator{\supp}{supp}

\title[Stationary distributions for jump processes]{Stationary distributions for jump processes with memory}

\author{K.\ Burdzy,  T.\ Kulczycki \and R.L.\ Schilling}
\address{Krzysztof Burdzy, Department of Mathematics, Box 354350, University of Washington, Seattle, WA 98195, USA}
\email{burdzy@math.washington.edu}
\thanks{K. Burdzy was supported in part by NSF Grant DMS-0906743 and by grant N N201 397137, MNiSW, Poland.}
\address{Tadeusz Kulczycki, Institute of Mathematics, Polish Academy of Sciences, ul.\ Kopernika 18, 51-617 Wroc{\l}aw, Poland \newline Institute of Mathematics and Computer Science, Wroc{\l}aw University of Technology, Wybrzeze Wyspianskiego 27, 50-370 Wroc{\l}aw, Poland}
\email{t.kulczycki@impan.pl}
\thanks{T. Kulczycki was supported in part by grant N N201 373136, MNiSW, Poland.}
\address{Rene Schilling, Institut f\"ur Stochastik, TU Dresden, D-01062 Dresden, Germany.}
\email{rene.schilling@tu-dresden.de}
\thanks{ R.L.\ Schilling was supported in part by DFG grant Schi 419/5-1.}

\begin{document}

\begin{abstract}
We analyze a jump processes $Z$ with a jump measure determined by  a ``memory'' process $S$. The state space of $(Z,S)$ is the  Cartesian product of the unit circle and the real line. We prove that the stationary distribution of $(Z,S)$ is the product of the uniform probability measure and a Gaussian distribution.
\end{abstract}

\maketitle

\section{Introduction}\label{sec1}

We are going to find stationary distributions for processes with jumps influenced by ``memory''. This paper is a companion to \cite{BKS}. The introduction to that paper contains a review of various sources of inspiration for this project, related models and results.

We will analyze a pair of real-valued processes $(Y,S)$ such that $S$ is a ``memory'' in the sense that $dS_t  =   W(Y_t) \,dt$ where $W$  is a $C^3$ function. The process
$Y$ is a jump process ``mostly'' driven by a stable process but the process $S$ affects the rate of jumps of $Y$. We refer the reader to Section \ref{sec2} for a formal presentation of this model as it is too long for the introduction. The present article illustrates advantages of semi-discrete models introduced in \cite{BW}  since the form of the  stationary distribution for $(Y,S)$ was conjectured in \cite[Example 3.8]{BW}. We would not find it easy to conjecture the stationary distribution for this process in a direct way.

The main result of this paper, i.e.\  Theorem \ref{uniq2}, is concerned with the stationary distribution of a transformation of $(Y,S)$. In order to obtain non-trivial results, we ``wrap'' $Y$ on the unit circle, so that the state space for the transformed process is compact. In other words, we consider
$(Z_t,S_t) = (e^{iY_t},S_t)$. The stationary distribution for
$(Z_t,S_t)$ is the product of the uniform distribution on the circle and the normal distribution.

The Gaussian distribution of the ``memory'' process appeared in models discussed in \cite{BBCH,BKS}. In each of those papers, memory processes similar to $S$ effectively represented ``inert drift''. A heuristic argument given in the introduction to \cite{BKS} provides a justification for the Gaussian distribution, using the concepts of kinetic energy associated to drift and Gibbs measure. The conceptual novelty of the present paper is that the Gaussian distribution of $S$ in the stationary regime cannot be explained by kinetic energy because $S$ affects the jump distribution and not the drift of $Z$.

The product form of the stationary distribution for a two-component Markov process is obvious if the two components are independent Markov processes. The product form is far from obvious if the components are not independent but it does appear in a number of contexts, from queuing theory to mathematical physics. The paper \cite{BW} was an attempt to understand this phenomenon for a class of models. The unexpected appearance of the Gaussian distribution in some stationary measures was noticed in \cite{BW2004} before it was explored more deeply in \cite{BW,BBCH}.

We turn to the technical aspects of the paper.
The main effort is directed at determining the domain and a core of the generator of the process.
A part of the argument is based on an estimate of the smoothness of the stochastic flow of solutions to
\eqref{rep12}.

\subsection{Notation}\label{notation}
Since the paper uses a large amount of notation, we collect  some of the most frequently used symbols in the table below, for easy reference.

\bigskip\noindent
\begin{center}
\begin{footnotesize}
\begin{tabular}{l|p{10cm}}
$a\vee b$, $a\wedge b$
            & $\max(a,b)$, $\min(a,b)$;\\
\hline \\
$a_+$, $a_-$
            & $\max(a,0)$, $-\min(a,0)$;\\ \hline \\
$|x|_{\ell^1}$
            & $\displaystyle \sum_{j=1}^m |x_j|$ where $x=(x_1,\ldots, x_m)\in\R^m$;\\ \hline \\
$e_k$ & the $k$-th unit base vector in the usual orthonormal basis for $\R^n$;\\ \hline \\
$\calA_{\alpha}$
            & $\displaystyle \alpha \Gamma\left(\frac{1 + \alpha}2\right) \frac{2^{\alpha - 1}}{\sqrt\pi\,\Gamma\big(1-\frac\alpha2\big)}$, $\alpha\in (0,2)$;\\ \hline \\
$D^\alpha$
            & $\displaystyle \frac{\partial^{|\alpha|}}{\partial x_1^{\alpha_1} \cdots \partial x_d^{\alpha_d}}$, $\alpha=(\alpha_1,\ldots,\alpha_d)\in\N_0^d$;\\ \hline \\
$C^k$       & $k$-times continuously differentiable functions;\\ \hline \\
$C^k_b$, $C^k_c$, $C^k_0$
            & functions in $C^k$ which, together with all their derivatives up to order $k$, are ``bounded'', are ``compactly supported'', and ``vanish at infinity'',  respectively;\\ \hline \\

$C_*(\Rt)$  &  all bounded and uniformly continuous functions $f:\Rt \to \R$ such that
            $\supp(f) \subset \R \times [-N,N]$ for some $N>0$;\\ \hline \\

$C_*(\Rt)$  &  $C_*(\Rt) \cap C_b^2(\Rt)$; \\ \hline \\

$\uni$      &  $\{z\in \C \::\: |z|=1\}$ unit circle in $\C$.
\end{tabular}
\end{footnotesize}
\end{center}

\medskip\noindent
Constants $c$ without sub- or superscript are generic and may change their value from line to line.

\section{The construction of the process and its generator}\label{sec2}

Let $\uni = \{z \in \C \::\: |z| = 1\}$ be the unit circle in $\C$. Consider a $C^3$ function $V: \uni \to \R$ such that $\int_\uni
V(z) \, dz = 0$ and set $W(x) = V(e^{ix})$, $x \in \R$.
Assume that $V$ is not identically constant.
In this paper we will be interested in the Markov process
$(Y_t,S_t)$ with state space $\Rt$ and generator $\calG^{(Y,S)}$ of
the following form
\begin{equation}\label{GYS}
    \calG^{(Y,S)} f(y,s) = -(-\Delta_y)^{\alpha/2} f(y,s) + Rf(y,s) + W(y) f_s(y,s),
\end{equation}
with a domain that will be specified later.
Here, $(y,s) \in \Rt$, $\alpha \in (0,2)$ and
\begin{gather}\notag
    -(-\Delta_y)^{\alpha/2} f(y,s)
    = \calA_{\alpha} \lim_{\varepsilon \to 0^+} \int_{|y-x| > \varepsilon} \frac{f(x,s) - f(y,s)}{|y - x|^{1 + \alpha}} \, dx,\\
\label{defR1}
    Rf(y,s) = \int_{-\pi + y}^{\pi+y} \big(f(x,s) - f(y,s)\big) \big((W(y) - W(x))s\big)_{+} \, dx.
\end{gather}
Since $-(-\Delta)^{\alpha/2}$, $\alpha \in (0,2)$, is the generator of the symmetric $\alpha$-stable process on $\R$, we may think of the process $Y_t$ as the perturbed symmetric $\alpha$-stable process and $S_t$ as the memory which changes the jumping measure of the process $Y_t$.

The definition of $(Y,S)$ is informal. Below we will construct this process in a direct way and we will show that this process has the generator \eqref{GYS}; see Proposition \ref{genYS2}. Our construction is based on the so-called \emph{construction of Meyer}; see, e.g., \cite{M} or \cite[Section 3.1]{BGK}.

For any $(y,s) \in \Rt$ let
$$
    g(y,s,x) = ((W(y)-W(y+x))s)_{+}\, \I_{(-\pi,\pi)}(x), \quad x \in \R,
$$
and
$$
    \|g(y,s,\cdot)\|_1 = \int_{-\pi}^{\pi} ((W(y)-W(y + x))s)_{+} \, dx.
$$
Let  $\ol g(y,s,x) := g(y,s,x)/\|g(y,s,\,\cdot)\|_1$  if $\|g(y,s,\cdot)\|_1\ne 0$.
We let $\ol g(y,s,\,\cdot\,)$ be the delta function at 0 when
$\|g(y,s,\cdot)\|_1= 0$. If $\|g(y,s,\cdot)\|_1\ne 0$, we let
$F_{y,s} (\,\cdot\,)$ denote the cumulative distribution function of a random variable with density $\ol g(y,s,\,\cdot\,)$. If $\|g(y,s,\cdot)\|_1 = 0$, we let
$F_{y,s} (\,\cdot\,)$ denote the cumulative distribution function of a random variable that is identically equal to 0.
We have
\begin{align*}
    F_{y,s}^{-1}(v) = \inf\left\{x\in \R: \int_{-\infty}^x  \frac{g(y,s,z)}{\|g(y,s,\,\cdot,)\|_1}  dz \geq v\right\}
\end{align*}
so for any $v$, the function $(y,s)\to F_{y,s}^{-1}(v)$ is measurable.
If $\unif$ is a uniformly distributed random variable on $(0,1)$, then $F_{y,s}^{-1}(\unif)$ has the density $\ol g(y,s,\,\cdot\,)$. Let $(\unif_n)_{n\in\N}$ be countably many independent copies of  $\unif$ and set $\eta_n(y,s) = F_{y,s}^{-1}(\unif_n)$.

Let $X(t)$ be a symmetric $\alpha$-stable process on $\R$, $\alpha \in (0,2)$, starting from $0$ and $N(t)$ a Poisson process with intensity $1$. We assume that
$(\unif_n)_{n \in \N}$, $X(\cdot)$ and $N(\cdot)$ are independent.

Let $0 < \sigma_1 < \sigma_2 <  \ldots$ be
the times of jumps of $N(t)$.
Consider any $y,s \in \R$ and for $t\geq 0$ let
\begin{align*}
    Y^1_t &= y + X_t, \\
    S^1_t &= s + \int_0^t W(Y^1_r) \, dr,\\
   \wh \sigma_1(t) &= \int_{0}^t \|g(Y^1_r,S^1_r,\cdot)\|_1 \, dr, \\
    \tau_1 &= \inf_{t \ge 0} \{\wh\sigma_1(t) = \sigma_1\}, \qquad (\inf\emptyset = \infty).
\end{align*}

Now we proceed recursively. If $Y^j_t$, $S^j_t$, $\wh\sigma_j(t)$ are well defined on $[0,\tau_j)$ and $\tau_j < \infty$ then we define for $t\geq \tau_j$,
\begin{align*}
    Y^{j+1}_t &= y + X_t + \sum_{n = 1}^{j} \eta_n(Y^n(\tau_{n}-),S^n(\tau_n-))\\
    S^{j+1}_t &= s + S^j(\tau_j-)
           +  \int_{\tau_j-}^t  W(Y^{j+1}_r) \, dr,\\
   \wh \sigma_{j+1}(t) &= \tau_j
  + \int_{\tau_j}^t \|g(Y^{j+1}_r,S^{j+1}_r,\cdot)\|_1 \, dr, \\
    \tau_{j+1} &= \inf_{t \ge \tau_j} \{\wh\sigma_{j+1}(t) = \sigma_{j+1}\}.
\end{align*}
Let $\tau_0 =0$
$(Y_t,S_t) = (Y^j_t, S^j_t)$ for $ \tau_{j-1} \leq t < \tau_j$, $j\geq 1$. It is easy to see that
$(Y_t,S_t)$ is defined for all $t\geq 0$, a.s.
If we put $\sigma(t) = \int_{0}^t \|g(Y_r,S_r,\cdot)\|_1 \, dr $ then we can represent $(Y_t,S_t)$ by the following closed-form expression,
\begin{equation}\label{rep12}\left\{\begin{aligned}
    Y_t &= y + X_t + \sum_{n = 1}^{N(\sigma(t))} \eta_n(Y(\tau_{n}-),S(\tau_n-)),\\
    S_t &= s + \int_0^t W(Y_r) \, dr.
\end{aligned}\right.\end{equation}

We define the semigroup $\{T_t\}_{t \ge 0}$ of the process $(Y_t,S_t)$ for $f \in C_b(\R^2)$ by
\begin{equation*}
    T_tf(y,s) = \Ee^{(y,s)} f(Y_t,S_t), \quad (y,s) \in \Rt.
\end{equation*}
By $\calG^{(Y,S)}$ we denote the generator of $\{T_t\}_{t \ge 0}$ and its domain by $\calD(\calG^{(Y,S)})$. We will show in Proposition \ref{genYS2} that $C_*^2(\R^2) \subset \calD(\calG^{(Y,S)})$ and that $\calG^{(Y,S)} f$ is given by \eqref{GYS} for $f \in C_*^2(\R^2)$, see Subsection \ref{notation} for the definition of $C_*^2(\R^2)$.

Our construction of $(Y_t,S_t)$ is a deterministic map
\begin{align*}
\left\{ (\unif_n)_{n\in \N}, (N(t))_{t\geq 0}, (X(t))_{t\geq 0}\right\}
\xrightarrow{\qquad}
\left\{ (Y(t))_{t\geq 0}, (S(t))_{t\geq 0}\right\}.
\end{align*}
This easily implies the strong Markov property for $(Y,S)$.
We will
 verify that
$(Z_t,S_t):=(e^{iY_t},S_t)$
 is also
a strong Markov
 process. We first show that the transition function of $(Y_t,S_t)$ is periodic.
\begin{lemma}\label{sol-periodic2}
    Let $(Y_t,S_t)$ be the Markov process defined by \eqref{rep12}. Then
    $$
        \Pp^{(y + 2 \pi,s)}(Y_t \in A + 2 \pi, \, S_t \in B) = \Pp^{(y,s)}(Y_t \in A, \, S_t \in B),
    $$
for all $(y,s)\in\R^2$ and all Borel sets $A,B\subset\R$.
\end{lemma}
\begin{proof}
    Let $X_t$ be a symmetric $\alpha$-stable process,  starting from $0$,  $\alpha\in (0,2)$, and let $N(t)$ be a Poisson process with intensity $1$. By $(Y_t^y,S_t^s)$ we denote the process given by \eqref{rep12} with initial value $(Y_0^y,S_0^s)=(y,s)$.
The process $(\tilde Y_t, \tilde S_t) := (Y_t^{y+2\pi},S_t^s)$
 has the following representation
    \begin{align*}
        \tilde{Y}_t & = y + 2 \pi + X_t + \sum_{n = 1}^{N(\tilde{\sigma}(t))} \eta_n(\tilde{Y}(\tilde{\tau}_{n}-),\tilde{S}(\tilde{\tau}_n-)),\\
        \tilde{S}_t & = s + \int_0^t W(\tilde{Y}_r) \, dr,
    \end{align*}
where $\tilde{\sigma}(t) = \int_{0}^t \|g(\tilde{Y}_r,\tilde{S}_r,\cdot)\|_1 \, dr$ and $\tilde{\tau}_k = \inf_{t \ge 0} \{\tilde{\sigma}(t) = \sigma_k\}$.

    Note that for all $x\in\R$,
    \begin{gather*}
        g(y - 2 \pi,s,x) = g(y,s,x)
\quad\text{and, therefore,}\quad
         \|g(y - 2 \pi,s,\cdot)\|_1 = \|g(y,s,\cdot)\|_1.
    \end{gather*}
It follows that
 $\eta_n(y - 2\pi,s)$ has the same distribution as $\eta_n(y,s)$. Since the function $W$ is periodic with period $2 \pi$, we have $W(\tilde{Y}_r) = W(\tilde{Y}_r - 2 \pi)$. Moreover, $\|g(\tilde{Y}_r,\tilde{S}_r,\cdot)\|_1 =  \|g(\tilde{Y}_r - 2 \pi,\tilde{S}_r,\cdot)\|_1$ and,  $\eta_n(\tilde{Y}(\tilde{\tau}_{n}-),\tilde{S}(\tilde{\tau}_n-))$ has the same distribution as $\eta_n(\tilde{Y}(\tilde{\tau}_{n}-)- 2 \pi,\tilde{S}(\tilde{\tau}_n-))$. This means that we can rewrite the representation of  $(Y_t^{y+2\pi},S_t^s)$ in the following way:
    \begin{align*}
        \tilde{Y}_t & = y + 2 \pi + X_t + \sum_{n = 1}^{N(\tilde{\sigma}(t))} \eta_n(\tilde{Y}(\tilde{\tau}_{n}-) - 2 \pi,\tilde{S}(\tilde{\tau}_n-)),\\
        \tilde{S}_t & = s + \int_0^t W(\tilde{Y}_r - 2 \pi) \, dr,
    \end{align*}
    where $\tilde{\sigma}(t)  = \int_{0}^t \|g(\tilde{Y}_r - 2 \pi,\tilde{S}_r,\cdot)\|_1 \, dr$ and $\tilde{\tau}_k = \inf_{t \ge 0} \{\tilde{\sigma}(t) = \sigma_k\}$.

    By subtracting $2 \pi$ from both sides of the first equation we get
    \begin{align*}
        \tilde{Y}_t - 2 \pi & = y + X_t + \sum_{n = 1}^{N(\tilde{\sigma}(t))} \eta_n(\tilde{Y}(\tilde{\tau}_{n}-) - 2 \pi,\tilde{S}(\tilde{\tau}_n-)),\\
        \tilde{S}_t & = s + \int_0^t W(\tilde{Y}_r - 2 \pi) \, dr,
    \end{align*}
    with $\tilde{\sigma}(t)$ and $\tilde{\tau}_k$ as before. Substituting $\hat Y_t := \tilde Y_t - 2\pi$ we see that this is the defining system of equations for the process $(Y^y_t,S^s_t)$. Therefore, the processes $(Y_t^y,S_t^s)$ and $(Y_t^{y+2\pi},S_t^s)$ have the same law.
\end{proof}

We can now argue exactly as in \cite[Corollary 2.3]{BKS} to see that $(Z_t,S_t)=(e^{iY_t},S_t)$ is indeed a strong Markov process.  We define the transition semigroup of $(Z_t,S_t)$ for $f \in C_0(\uni \times \R)$ by
\begin{equation}\label{TtD1}
    T_t^\uni f(z,s) = \Ee^{(z,s)} f(Z_t,S_t), \quad (z,s) \in \uni \times \R.
\end{equation}

The generator of $\{T_t^\uni\}_{t \ge 0}$ and its domain will be denoted $\calG$ and $\calD(\calG)$.

In the sequel we will need the following auxiliary processes
\begin{align*}
    \hat{Y}_t &= \hat{Y}_0 + X_t, \\
    \hat{S}_t &= \hat{S}_0 + \int_0^t W(\hat{Y}_r) \, dr, \\
    \hat{Z}_t &= e^{i \hat{Y}_t},
\end{align*}
where $X_t$ is a symmetric $\alpha$-stable L\'evy process on $\R$, $\alpha \in (0,2)$, starting from $0$.
We will use the following notation:

$$
\begin{array}{l||l|l}
\vphantom{\displaystyle\int}
\text{Process} & \text{Semigroup} &\text{Generator and domain}\\\hline
\vphantom{\displaystyle\int\limits_0^1}
(Y_t,S_t)
    & T_t, t\geq 0
    & \big(\calG^{(Y,S)},\calD(\calG^{(Y,S)})\big)\\\hline
\vphantom{\displaystyle\int\limits_0^1}
(Z_t,S_t) = (e^{iY_t},S_t)
    & T^\uni_t, t\geq 0
    & \big(\calG,\calD(\calG)\big)\\\hline
\vphantom{\displaystyle\int\limits_0^1}
(\hat{Y}_t,\hat{S}_t)=\big(\hat Y_0+X_t,\hat S_0+\int_0^t W(\hat Y_r)\,dr\big)
    & \hat{T}_t, t\geq 0
    & \big(\calG^{(\hat{Y},\hat{S})}, \calD(\calG^{(\hat{Y},\hat{S})})\big)\\\hline
\vphantom{\displaystyle\int\limits_0^1}
(\hat{Z}_t,\hat{S}_t) = (e^{i\hat Y_t},\hat S_t)
    & \hat{T}_t^\uni, t\geq 0
    & \big(\hat\calG,\calD(\hat{\calG})\big)
\end{array}
$$

We will now identify the generators of the processes
$(Y_t,S_t)$ and
 $(Z_t,S_t)$ and link them with the generators of the processes $(\hat{Y}_t,\hat{S}_t)$ and
 $(\hat{Z}_t,\hat{S}_t)$.
\begin{proposition}\label{genYS1}
    Let $(Y_t,S_t)$ be the process defined by \eqref{rep12} and let $f \in C_*(\Rt)$. Then
    $$
        \lim_{t \to 0^+} \frac{T_t f - f}{t} \quad \text{exists} \quad \iff \quad
        \lim_{t \to 0^+} \frac{\hat{T}_t f - f}{t} \quad \text{exists},
    $$
in the norm $\|\cdot\|_{\infty}$.
If one, hence both, limits exist, then
    \begin{equation}\label{new-eq}
        \lim_{t \to 0^+} \frac{T_t f - f}{t} = \lim_{t \to 0^+} \frac{\hat{T}_t f - f}{t} + Rf,
    \end{equation}
    where $Rf$ is given by \eqref{defR1}.
\end{proposition}

\begin{corollary}\label{genZS1}
    We have
    $$
        f \in \calD(\calG) \cap C_c(\uni \times \R)
        \iff
        f \in \calD(\hat{\calG}) \cap C_c(\uni \times \R).
    $$
    If $f \in \calD(\calG) \cap C_c(\uni \times \R)$ then
    $$
        \calG f = \hat{\calG} f + R^\uni f,
    $$
    where
    $$
        R^\uni f(z,s) = \int_{\uni} (f( w,s) - f(z,s)) ((V(z) - V( w))s)_{+} \, d w.
    $$
\end{corollary}

\begin{proposition}\label{genYS2}
    Let $(Y_t,S_t)$ be the process defined by \eqref{rep12}. Then $C_*^2(\Rt) \subset \calD(\calG^{(Y,S)})$ and for $f \in C_*^2(\Rt)$ we have
    \begin{equation}\label{formulaYS2a}
        \calG^{(Y,S)} f(y,s) = -(-\Delta_y)^{\alpha/2} f(y,s) + Rf(y,s) + W(y) f_s(y,s),
    \end{equation}
    for all $(y,s) \in \Rt$ with $Rf$ given by \eqref{defR1}.

    Moreover, $C_*^2(\Rt) \subset \calD(\calG^{(\hat{Y},\hat{S})})$ and for $f \in C_*^2(\Rt)$ we have
    \begin{equation}\label{formulaYS2b}
        \calG^{(\hat{Y},\hat{S})} f(y,s) = -(-\Delta_y)^{\alpha/2} f(y,s) + W(y) f_s(y,s),
    \end{equation}
    for all $(y,s) \in \Rt$.
\end{proposition}

By $\Arg(z)$ we denote the argument of $z \in \C$ contained in $(-\pi,\pi]$. For $g \in C^2(\uni)$ let us put
\begin{equation}\label{Lz1}\begin{aligned}
    Lg(z)
    &=  \calA_{\alpha} \lim_{\eps \to 0^+} \int_{\uni \cap \{|\Arg(w/z)| > \eps\}} \frac{g(w) - g(z)}{|\Arg(w/z)|^{1 + \alpha}} \, dw \\
    &\qquad+\calA_{\alpha} \sum_{n \in \Z \setminus \{0\}} \int_{\uni} \frac{g(w) - g(z)}{|\Arg(w/z) + 2 n \pi|^{1 + \alpha}} \, dw,
\end{aligned}\end{equation}
where $dw$ denotes the
arc length measure on $\uni$; note that $\int_\uni \,dw = 2 \pi$.
 It is clear that for $f \in C_c^2(\uni \times \R)$, $z= e^{iy}$, $y,s \in \R$ we have
\begin{equation}
\label{Lz3}
-(-\Delta_y)^{\alpha/2}\tilde{f}(y,s) = L_z f(z,s).
\end{equation}

\begin{corollary}\label{genZS2}
    We have $C_c^2(\uni \times \R) \subset \calD(\calG)$ and for $f \in C_c^2(\uni \times \R)$ we have
    $$
        \calG f(z,s) = L_z f(z,s) + R^\uni f(z,s) + V(z) f_s(z,s),
    $$
    for all $(z,s) \in \uni \times \R$, where $L$ is given by  \eqref{Lz1}.

    We also have  $C_c^2(\uni \times \R) \subset \calD(\hat{\calG})$ and for $f \in C_c^2(\uni \times \R)$ we have
    $$
        \hat{\calG} f(z,s) = L_z f(z,s) + V(z) f_s(z,s),
    $$
    for all $(z,s) \in \uni \times \R$.
\end{corollary}

\begin{remark}\label{rem1}
    Proposition \ref{genYS2} shows that for  $f \in C_*^2(\Rt)$  the generator of the process $(Y_t,S_t)$ defined by \eqref{rep12} is of the form \eqref{GYS}. This is a standard result, the so-called ``construction of Meyer'', but we include our own proof of this result so that the paper is self-contained.  Moreover, Proposition \ref{genYS1}, Corollaries \ref{genZS1} and \ref{genZS2} are needed to identify a core for $\calG$. Corollary \ref{genZS2} is also needed to find the stationary measure for $(Z_t,S_t)$.
\end{remark}

We will need two auxiliary results.
\begin{lemma}\label{functiong}
    There exists a constant $c = c( M) > 0$ such that for any $x \in [-\pi,\pi]$ and any $u_1 = (y_1,s_1) \in \Rt$, $u_2 = (y_2,s_2) \in \Rt$ with $s_1,s_2 \in [- M, M]$ we have
    $$
        |g(u_1,x) - g(u_2,x)| \le c\, (|u_2 - u_1| \wedge 1).
    $$
\end{lemma}
\begin{proof}
From $a_+ = (a + |a|)/2$ we conclude that $|a_+ - b_+| \le |a - b|$ for all $a,b \in \R$.

Let $x \in [-\pi,\pi]$, $u_1 = (y_1,s_1) \in \Rt$, $u_2 = (y_2,s_2) \in \Rt$  and $s_1,s_2 \in [- M, M]$.  We have
\begin{align*}
    |g(u_1&,  x) - g(u_2,x)| \\
    &\le |((W(y_1) -W(y_1 + x)) s_1)_{+} - ((W(y_2) -W(y_2 + x)) s_2)_{+}| \\
    &\le |(W(y_1) -W(y_1 + x)) s_1 - (W(y_2) -W(y_2 + x)) s_2| \\
    &\le |(W(y_1) -W(y_1 + x)) s_1 - (W(y_1) -W(y_1 + x)) s_2| \\
    &\qquad+ |(W(y_1) -W(y_1 + x)) s_2 - (W(y_2) -W(y_2 + x)) s_2| \\
    &\le |W(y_1) -W(y_1 + x)| |s_1 - s_2| + |W(y_1) -W(y_2)| |s_2| \\
    &\qquad+ |W(y_1 +x) -W(y_2 + x)| |s_2| \\
    &\le 2 \|W\|_{\infty} |s_1 - s_2| + 2  M \|W'\|_{\infty} |y_1 - y_2|.
\end{align*}
Since, trivially, $|g(u_1,x) - g(u_2,x)| \le 4 \|W\|_{\infty}  M$, the claim follows with
$c= 4(\|W\|_\infty+\|W'\|_\infty)( M  +1)$.
\end{proof}

As an easy corollary of Lemma \ref{functiong} we get
\begin{lemma}\label{norm}
    There exists a constant $c = c( M) > 0$ such that for any $u_1 = (y_1,s_1) \in \Rt$, $u_2 = (y_2,s_2) \in \Rt$ with $s_1,s_2 \in [- M, M]$ we have
    $$
        \left| \, \|g(u_1,\cdot)\|_1 - \|g(u_2,\cdot)\|_1\right| \le c (|u_2 - u_1| \wedge 1).
    $$
\end{lemma}

\begin{proof}[Proof of Proposition \ref{genYS1}]
Let $f \in C_*(\Rt)$. Throughout the proof we will assume that $\supp(f) \subset \R \times (-M_0,M_0)$ for some $M_0 > 0$. Note that
$$
    |S_t|
    = \left|S_0 + \int_{0}^t W(Y_r) \, dr \right|
    \le |S_0| + \|W\|_{\infty}
    \le M_0 + \|W\|_{\infty}.
$$
for all starting points $(Y_0,S_0) = (y,s) \in \R \times [-M_0,M_0]$ and all $0 \le t \le 1$. Put
$$
    M_1 = M_0 + \|W\|_{\infty}.
$$
If $(Y_0,S_0) = (y,s) \notin \R \times [-M_1,M_1]$, then
$$
    |S_t| = \left|S_0 + \int_{0}^t W(Y_r) \, dr \right| > M_1 - \|W\|_{\infty} = M_0,
    \quad 0\leq t\leq 1,
$$
so
 $f(Y_t,S_t) = 0$. It follows that for any $(y,s) \notin \R \times [-M_1,M_1]$ and $0 < h \le 1$ we have
$$
    \frac{\Ee^{(y,s)}f(Y_h,S_h) - f(y,s)}{h} = 0.
$$
By the same argument,
$$
    \frac{\Ee^{(y,s)}f(\hat{Y}_h,\hat{S}_h) - f(y,s)}{h} = 0.
$$
It now follows from the definition of $R f(y,s)$ that $R f(y,s) = 0$ for $(y,s) \notin \R \times [-M_1,M_1]$.
 It is, therefore, enough to consider $(y,s) \in \R \times [-M_1,M_1]$.

The arguments above tell us that for all starting points $(Y_0,S_0) = (y,s) \in \R \times [-M_1,M_1]$ and all $0 \le t \le 1$,
$|S_t| \le  |S_0| + \|W\|_{\infty} \leq M_1 + \|W\|_{\infty}$. Setting
$$
     M = M_1 + \|W\|_{\infty},
$$
we get from the definition of the function $g$ that
$$
    \|g(Y_r,S_r,\cdot)\|_1 \le 2 \pi\, 2 \|W\|_{\infty}  M,\quad 0\leq r\leq 1,
$$
and so
$$
    \sigma(t)
    = \int_0^t \|g(Y_r,S_r,\cdot)\|_1 \, dr
    \le  4 \pi \|W\|_{\infty}  M t
    = c_0 t,
    \quad 0 \le t \le 1,
$$
with the constant $c_0 =  4 \pi \|W\|_{\infty}  M$.

From now on we will assume that $(y,s) \in \R \times [-M_1,M_1]$ and $0 < h \le 1$. We have
\begin{align*}
    \frac{T_hf(y,s) - f(y,s)}{h}
    &= \frac{\Ee^{(y,s)}f(Y_h,S_h) - f(y,s)}{h} \\
    &= \frac{1}{h} \Ee^{(y,s)}[f(Y_h,S_h) - f(y,s);\: N(\sigma(h)) = 0] \\
    &\qquad+ \frac{1}{h} \Ee^{(y,s)}[f(Y_h,S_h) - f(y,s);\: N(\sigma(h)) = 1] \\
    &\qquad+ \frac{1}{h} \Ee^{(y,s)}[f(Y_h,S_h) - f(y,s);\: N(\sigma(h)) \ge 2] \\
    &=  \text{I} + \text{II} + \text{III}.
\end{align*}

Since $\sigma(h) \le c_0 h$ we obtain
\begin{align*}
    |\text{III}|
    \leq \frac{2\|f\|_{\infty}}{h}\, \Pp^{(y,s)}[N(\sigma(h)) \ge 2]
    &\le \frac{2\|f\|_{\infty}}{h}\, \Pp^{(y,s)}[N(c_0 h) \ge 2] \\
    &= 2 \|f\|_{\infty}\, \frac{1 - e^{-c_0 h} - c_0 h e^{-c_0 h}}{h}
    \xrightarrow[h\to 0^+]{} 0
\end{align*}
uniformly for all  $(y,s) \in \R \times [-M_1,M_1]$.

Now we will consider the expression $\text{I}$. We have
\begin{align*}
    \text{I}
    &= \frac{1}{h}\, \Ee^{(y,s)}\left[ f\big(y + X_h, s+{\textstyle \int_0^h W(y + X_r) \, dr}\big) - f(y,s);\: N(\sigma(h)) = 0\right] \\
    &= \frac{1}{h}\, \Ee^{(y,s)}\big[f(\hat{Y}_h,\hat{S}_h) - f(y,s);\: N(\sigma(h)) = 0\big] \\
    &= \frac{1}{h}\, \Ee^{(y,s)}\big[f(\hat{Y}_h,\hat{S}_h) - f(y,s)\big]
    - \frac{1}{h}\, \Ee^{(y,s)}\big[f(\hat{Y}_h,\hat{S}_h) - f(y,s);\: N(\sigma(h)) \ge 1\big] \\
    &= \text{I}_1 + \text{I}_2.
\end{align*}
Note that
$$
    \text{I}_1 = \frac{\hat{T}_hf(y,s) - f(y,s)}{h}.
$$
It will suffice to prove that $\text{I}_2 \to 0$ and $\text{II} \to R f$. We have
\begin{align*}
    |\text{I}_2|
    &\le \frac{1}{h}\, \Ee^{(y,s)}\left[|f(\hat{Y}_h,\hat{S}_h) - f(y,s)|;\: N(c_0 h) \ge 1\right] \\
    &= \frac{1 - e^{-c_0 h}}{h} \, \Ee^{(y,s)}\left[|f(\hat{Y}_h,\hat{S}_h) - f(y,s)|\right].
\end{align*}

Recall that $f \in C_*(\Rt)$ is bounded and uniformly continuous. We will use the following modulus of continuity
$$
    \varepsilon(f;\delta) = \varepsilon(\delta) = \sup_{(y,s) \in \R^2} \sup_{|y_1| \vee |s_1| \le \delta} |f(y + y_1, s + s_1) - f(y,s)|.
$$
Clearly, $\varepsilon(\delta) \le 2 \|f\|_{\infty}$ and $\lim_{\delta \to 0^+} \varepsilon(\delta) = 0$.

Note that for $\hat{Y}_0 = y$, $\hat{S}_0 = s$ we have $\hat{Y}_h - y = X_h$, $\hat{S}_h - s =  \int_0^h W(\hat{Y}_r) \, dr $ which gives $|\hat{S}_t - s| \le h \|W\|_{\infty}$ for all $t \le h$. It follows that
$$
    \Ee^{(y,s)}\big[\big|f(\hat{Y}_h,\hat{S}_h) - f(y,s)\big|\big]
    \le
    \Ee^{(y,s)}\left[\varepsilon\left(\sup_{0 < t \le h}(|X_t| \vee h \|W\|_{\infty})\right)\right].
$$
Since $t\mapsto X_t$ is right-continuous and $X_0 \equiv 0$ we
have, a.s.,
$$
    \sup_{0 < t \le h}|X_t| \xrightarrow[h\to 0^+]{} 0
    \quad\text{and, therefore,}\quad
    \varepsilon\left(\sup_{0 < t \le h}(|X_t| \vee h \|W\|_{\infty})\right) \xrightarrow[h\to 0^+]{} 0.
$$
By the bounded convergence theorem
$$
    \Ee^{(y,s)}\left[\varepsilon\left(\sup_{0 < t \le h}(|X_t| \vee h \|W\|_{\infty})\right)\right] \xrightarrow[h\to 0^+]{} 0
$$
uniformly for all $(y,s) \in \R \times [-M_1,M_1]$ because the expression $\varepsilon(\sup_{0 < t \le h}|X_t| \vee h \|W\|_{\infty})$ does not depend on $(y,s)$. It follows that
$$
    |\text{I}_2| \xrightarrow[h\to 0^+]{} 0
$$
uniformly for all $(y,s) \in \R \times [-M_1,M_1]$.

Now we turn to $\text{II}$. We have
\begin{align*}
    \text{II}
    &= \frac{1}{h} \Ee^{(y,s)}[f(Y_h,S_h) - f(Y_h,s);\: N(\sigma(h)) = 1] \\
    &\qquad+ \frac{1}{h} \Ee^{(y,s)}[f(Y_h,s) - f(y,s);\: N(\sigma(h)) = 1]\\
    &= \text{II}_1 + \text{II}_2.
\end{align*}
Since $\sigma(h)\leq c_0 h$
\begin{align*}
    |\text{II}_1|
    &\le \frac{1}{h} \, \Ee^{(y,s)}\left[\left|f\left(Y_h, s + {\textstyle \int_0^h W(Y_r) \, dr}\right) - f(Y_h,s)\right|;\: N(c_0 h) \ge 1\right]\\
    &\le \frac{1}{h} \Ee^{(y,s)}\left[\varepsilon\left(\left|{\textstyle \int_0^h W(Y_r) \, dr }\right| \right);\: N(c_0 h) \ge 1\right]\\
    &\le \frac{1}{h} \, \Ee^{(y,s)}\big[\varepsilon\left(h \|W\|_{\infty}\right);\: N(c_0 h) \ge 1\big]\\
    &= \frac{1 - e^{-c_0 h}}{h} \, \Ee^{(y,s)}\big[\varepsilon\left(h \|W\|_{\infty}\right)\big]
    \xrightarrow[h\to 0^+]{} 0
\end{align*}
uniformly for all  $(y,s) \in \R \times [-M_1,M_1]$. It will suffice to show that $\text{II}_2 \to Rf$.

From now on we will use the following shorthand notation
$$
    U_t := (Y_t,S_t),
    \quad \hat U_t := (\hat Y_t,\hat S_t),
    \quad u := (y,s).
$$
We have
\begin{align*}
    \text{II}_2
    &= \frac{1}{h} \, \Ee^{(y,s)}\big[f(y + X_h + \eta_1(U_{\tau_1-}),s) - f(y + \eta_1(U_{\tau_1-}),s);\: N(\sigma(h)) \ge 1\big] \\
    &\qquad+ \frac{1}{h}\, \Ee^{(y,s)}\big[f(y + \eta_1(U_{\tau_1-}),s) - f(y,s);\: N(\sigma(h)) \ge 1\big] \\
    &\qquad- \frac{1}{h}\, \Ee^{(y,s)}\big[f(y + X_h + \eta_1(U_{\tau_1-}),s) - f(y,s);\: N(\sigma(h)) \ge 2\big]\\
    &= \text{II}_{2a} + \text{II}_{2b} + \text{II}_{2c}.
\end{align*}
Observe that
$$
    |\text{II}_{2c}| \le 2 \|f\|_{\infty}\, \frac{1 - e^{-c_0 h} - c_0 h e^{-c_0 h}}{h} \xrightarrow[h\to 0^+]{} 0
$$
and that the convergence is uniform in $(y,s) \in \R \times [-M_1,M_1]$.

Moreover,
\begin{align*}
    |\text{II}_{2a}|
    &\le \frac{1}{h}\, \Ee^{(y,s)}\big[|f(y + X_h + \eta_1(U_{\tau_1-}),s) - f(y + \eta_1(U_{\tau_1-}),s)|;\: N(c_0 h) \ge 1\big] \\
    &\le \frac{1}{h}\, \Ee^{(y,s)}\left[\varepsilon \left(\sup_{0 \le t \le h} |X_h|\right);\: N(c_0 h) \ge 1 \right] \\
    &= \frac{1 - e^{-c_0 h}}{h} \Ee^{(y,s)}\left[\varepsilon\left(\sup_{0 \le t \le h} |X_h|\right)\right]\\
    &\xrightarrow[h\to 0^+]{} 0
\end{align*}
uniformly for all  $(y,s) \in \R \times [-M_1,M_1]$.
It will suffice to show that $\text{II}_{2b} \to Rf$.

Note that
\begin{equation}
\label{tauh1}
    N(\sigma(h)) \ge 1
    \iff \tau_1 \le h
    \iff \int_0^h \|g(U_r,\cdot)\|_1 \, dr \ge \sigma_1.
\end{equation}

We claim that
\begin{equation}
\label{tauh2}
    \int_0^h \|g(U_r,\cdot)\|_1 \, dr \ge \sigma_1
        \iff \int_0^h \|g(\hat{U}_r,\cdot)\|_1 \, dr \ge \sigma_1.
\end{equation}
First, we assume that $\int_0^h \|g(U_r,\cdot)\|_1 \, dr \ge \sigma_1$. This implies that $\tau_1  \le h$. Recall that $U_r = \hat{U}_r$ for $r < \tau_1$. Hence
$$
    \int_0^h \|g(\hat{U}_r,\cdot)\|_1 \, dr
        \ge \int_0^{\tau_1} \|g(\hat{U}_r,\cdot)\|_1 \, dr
        = \int_0^{\tau_1} \|g(U_r,\cdot)\|_1 \, dr
        = \sigma_1,
$$
where the last equality follows from the definition of $\tau_1$.

Now let us assume that $\int_0^h \|g(U_r,\cdot)\|_1 \, dr < \sigma_1$. This implies that $\tau_1  > h$. Using again $U_r = \hat{U}_r$ for $r < h < \tau_1$, we obtain
$$
    \sigma_1
        > \int_0^h \|g(U_r,\cdot)\|_1 \, dr
        = \int_0^h \|g(\hat{U}_r,\cdot)\|_1 \, dr,
$$
which finishes the proof of (\ref{tauh2}).

By (\ref{tauh1}) and (\ref{tauh2}) we obtain
\begin{align*}
    \text{II}_{2b}
    &= \frac{1}{h}\, \Ee^{(y,s)}\big[f(y + \eta_1(U_{\tau_1-}),s) - f(y,s);\: \tau_1 \le h\big] \\
    &= \frac{1}{h}\, \Ee^{(y,s)}\big[f(y + \eta_1(U_{(\tau_1 \wedge h)-}),s) - f(y,s);\: \tau_1 \le h\big] \\
    &= \frac{1}{h}\, \Ee^{(y,s)}\left[ f(y + \eta_1(\hat U_{(\tau_1 \wedge h)-}),s) - f(y,s);\: {\textstyle \int_0^h \|g(\hat{U}_r,\cdot)\|_1 \, dr } \ge \sigma_1 \right].
\end{align*}

We will use the following abbreviations:
\begin{gather*}
    u = (y,s),\\
    A = \left\{\int_0^h \|g(\hat{U}_r,\cdot)\|_1 \, dr \ge \sigma_1\right\},\\
    B = \left\{\int_0^h \|g(u,\cdot)\|_1 \, dr \ge \sigma_1\right\},\\
    F = \frac{1}{h} \big(f(y + \eta_1(\hat U_{(\tau_1 \wedge h)-}),s) - f(y,s)\big).
\end{gather*}
This allows us to rewrite $\text{II}_{2b}$ as
\begin{align*}
    \text{II}_{2b}
    = \Ee^u[F;A]
    &= \Ee^u[F;B] + \Ee^u[F,A \setminus B] - \Ee^u[F;B \setminus A].
\end{align*}

Recall that $X=(X_t)_{t\geq 0}$, $N=(N(t))_{t\geq 0}$ and $\unif = (\unif_n)_{n\in \N}$ are independent. Therefore the probability measure $\Pp$ can be written in the form $\Pp = \Pp_X\otimes \Pp_N \otimes \Pp_\unif$;  the conditional probability, given $N$ or $\unif$, is $\Pp_X$ and the corresponding expectation is denoted by $\Ee_X$. In a similar way $\Pp_{(X,N)}=\Pp_X\otimes\Pp_N$ and $\Ee_{(X,N)}$ denote conditional probability and conditional expectation if $\unif$ is given.  As usual, the initial (time-zero) value of the process under consideration is given as a superscript. Note that $\hat{U}_t = (\hat{Y}_t,\hat{S}_t)$ is a function of $X$ and does not depend on $N$ or $\unif$. In particular, $\hat{U}_t$ and $\sigma_1$ are independent. Since $\sigma_1$ is the time of the first jump of the Poisson process $N(t)$, it is exponentially distributed with parameter $1$.  It follows that
\begin{align*}
    &|\Ee^u[F,A \setminus B]|\\
    &\le  \frac{2 \|f\|_{\infty}}{h} \, \Pp^u \left[{\textstyle \int_0^h \|g(\hat{U}_r,\cdot)\|_1 \, dr } \ge \sigma_1 > {\textstyle \int_0^h \|g(u,\cdot)\|_1 \, dr}\right] \\
    &= \frac{2 \|f\|_{\infty}}{h} \, \Ee_X^u \left[e^{-\int_0^h \|g(u,\cdot)\|_1 \, dr} - e^{- \int_0^h \|g(\hat{U}_r,\cdot)\|_1 \, dr};\: {\textstyle \int_0^h \|g(\hat{U}_r,\cdot)\|_1 \, dr  > \int_0^h \|g(u,\cdot)\|_1 \, dr}\right] \\
    &\le \frac{2 \|f\|_{\infty}}{h} \Ee_X^u \left| e^{-\int_0^h \|g(u,\cdot)\|_1 \, dr} - e^{- \int_0^h \|g(\hat{U}_r,\cdot)\|_1 \, dr}\right| \\
    &\leq \frac{2 \|f\|_{\infty}}{h}\, \Ee_X^u \left| \int_0^h \left(\|g(u,\cdot)\|_1  -  \|g(\hat{U}_r,\cdot)\|_1 \right) dr\right| \\
    &\le 2 \|f\|_{\infty}\, \Ee_X^u \sup_{0 \le r \le h} \left| \|g(u,\cdot)\|_1  -  \|g(\hat{U}_r,\cdot)\|_1 \right|.
\end{align*}
For the penultimate inequality we used the elementary estimate $|e^{-a} - e^{-b}| \le |a - b|$, $a,b \ge 0$. From Lemma \ref{norm} we infer that the last expression is bounded by
\begin{align*}
    2 \|f\|_{\infty}\, c\, & \Ee_X^u\left[\sup_{0 \le r \le h} \big(|\hat{U}_r - u| \wedge 1\big)\right] \\
    &= 2 \|f\|_{\infty} \,c\, \Ee_X^u\left[\sup_{0 \le r \le h} \left(\left|\left(X_r,\int_0^r  W(\hat{Y}_t)  \, dt\right)\right| \wedge 1\right)\right]\\
    &\xrightarrow[h\to 0^+]{} 0
\end{align*}
uniformly for all  $(y,s) \in \R \times [-M_1,M_1]$. This convergence follows from the right-continuity of $X_r$ and the fact that $|\int_0^r  W(\hat{Y}_t)  \, dt| \le h \|W\|_{\infty}$.

A similar argument shows that
$\displaystyle |\Ee^u[F;B \setminus A]| \xrightarrow[h\to 0^+]{} 0$ uniformly in $(y,s) \in \R \times [-M_1,M_1]$.
It will suffice to show that $\Ee^u[F;B] \to Rf$.

We have
\begin{align*}
    \Ee^u[F;B]
    &= \frac{1}{h}\, \Ee^{(y,s)}\left[  f(y + \eta_1(\hat{U}_{(\tau_1 \wedge h)-}),s)  - f(y,s); \int_0^h \|g(u,\cdot)\|_1 \, dr \ge \sigma_1 \right] \\
    &= \frac{1}{h}\, \Ee^{(y,s)}\left[
     f(y + \eta_1(\hat{U}_{(\tau_1 \wedge h)-}),s) - f(y + \eta_1(u),s); \int_0^h \|g(u,\cdot)\|_1 \, dr \ge \sigma_1 \right] \\
    &\qquad+ \frac{1}{h} \Ee^{(y,s)}\left[ f(y + \eta_1(u),s) - f(y,s); \int_0^h \|g(u,\cdot)\|_1 \, dr \ge \sigma_1 \right] \\
    &= \textsf{A} + \textsf{B}.
\end{align*}

In order to deal with $\textsf{A}$ and $\textsf{B}$ we introduce the following auxiliary notation.

Recall that $X$, $N$ and $\unif$ are independent. As before let $\Ee_{(X,N)}^{(y,s)}$ be the conditional expectation given $\unif$; the superscript $(y,s)$ indicates that $Y_0 = y$ and $S_0 = s$.  Moreover, $\Ee_\unif$ denotes conditional expectation given $X$ and $N$.

\begin{lemma}\label{jump}
    Let $u_1 = (y_1,s_1) \in \R^2$, $u_2 = (y_2,s_2) \in \R^2$  be such that $s_1,s_2 \in [- M, M]$ and  $\|g(u_2,\cdot)\|_1 > 0$. Then we have
    $$
        |\Ee_\unif(f(y + \eta_1(u_1),s) - f(y + \eta_1(u_2),s))| \le c \left(\frac{|u_1 - u_2|}{\|g(u_2,\cdot)\|_1} \wedge 1\right),
    $$
    for some $c = c(f, M) > 0$.
\end{lemma}
\begin{proof}
We will distinguish two cases: $\|g(u_1,\cdot)\|_1 = 0$ and $\|g(u_1,\cdot)\|_1> 0$.

Assume that $\|g(u_1,\cdot)\|_1 = 0$.
Then by
 Lemma \ref{norm} we have
$$
    \|g(u_2,\cdot)\|_1 = \left|\|g(u_2,\cdot)\|_1 - \|g(u_1,\cdot)\|_1\right| \le c |u_2 - u_1|.
$$
Hence,
$$
    |\Ee_\unif(f(y + \eta_1(u_1),s) - f(y + \eta_1(u_2),s))| \le 2 \|f\|_{\infty}
    \le \frac{2\, \|f\|_{\infty} \,c\, |u_1 - u_2|}{\|g(u_2,\cdot)\|_1}.
$$

Now we will consider the second case: $\|g(u_1,\cdot)\|_1> 0$. We have
\begin{align*}
    &\big|\Ee_\unif(f(y + \eta_1(u_1),s) - f(y + \eta_1(u_2),s))\big| \\
    &= \left|\int_{-\pi}^{\pi} \frac{f(y+x,s)}{\|g(u_1,\cdot)\|_1} g(u_1,x) \, dx - \int_{-\pi}^{\pi} \frac{f(y+x,s)}{\|g(u_2,\cdot)\|_1} g(u_2,x) \, dx\right| \\
    &\le \frac{\displaystyle \left|\int_{-\pi}^{\pi} f(y + x,s) \big[g(u_1,x)\|g(u_2,\cdot)\|_1 - g(u_2,x)\|g(u_1,\cdot)\|_1\big]\, dx\right|}{\|g(u_1,\cdot)\|_1 \|g(u_2,\cdot)\|_1}  \\
    &\le \frac{\|f\|_{\infty}\, }{\|g(u_1,\cdot)\|_1 \|g(u_2,\cdot)\|_1} \left[\int_{-\pi}^{\pi}  \big|g(u_1,x) \|g(u_2,\cdot)\|_1 - g(u_1,x)\|g(u_1,\cdot)\|_1\big| \, dx \right. \\
    &\qquad \left. + \int_{-\pi}^{\pi}  \big|g(u_1,x) \|g(u_1,\cdot)\|_1 - g(u_2,x)\|g(u_1,\cdot)\|_1\big| \, dx\right]
\end{align*}
By Lemmas \ref{functiong} and \ref{norm} this is bounded from above by
\begin{gather*}
    \frac{\|f\|_{\infty}}{\|g(u_1,\cdot)\|_1 \|g(u_2,\cdot)\|_1} \left[\int_{-\pi}^{\pi} g(u_1,x) \, dx\,
    c' |u_2 - u_1| + 2 \pi c'' \, |u_1 - u_2|\, \|g(u_1,\cdot)\|_1 \right] \\
    \qquad\le \frac{(c' + 2 \pi c'')\, \|f\|_{\infty} \, |u_2 - u_1|}{\|g(u_2,\cdot)\|_1}.
\end{gather*}
The lemma follows now from the  observation   that
\begin{gather*}
    |\Ee_\unif(f(y + \eta_1(u_1),s)) - f(y + \eta_1(u_2),s))| \le 2 \|f\|_{\infty}.
\qedhere
\end{gather*}
\end{proof}

\begin{allowdisplaybreaks}
\bigskip\noindent
\emph{Proof of Proposition \ref{genYS1} (continued):}
We go back to $\textsf{A} + \textsf{B}$. If $\|g(u,\cdot)\|_1 = 0$ then $\textsf{A} + \textsf{B} = 0 = R f(y,s)$.
The proof of the proposition is complete in this case.

We will consider the case $\|g(u,\cdot)\|_1 > 0$. Because of the independence of $\sigma_1$,  $X_t$  and $(\eta_1(u))_{u \in \Rt}$ we get
\begin{equation*}
    |\textsf{A}|
    = \left|\frac{1}{h} \Ee_{(X,N)}^u\left[ \Ee_\unif \left[  f(y + \eta_1(\hat{U}_{(\tau_1 \wedge h)-}),s)  - f(y + \eta_1(u),s)\right];\: h \|g(u,\cdot)\|_1 \ge \sigma_1 \right] \right|.
\end{equation*}
By Lemma \ref{jump} this is bounded from above by
\begin{align*}
    &\Bigg| \frac{1}{h} \, \Ee_{(X,N)}^u\left[ c \left(\frac{|\hat{U}((\tau_1 \wedge h)-) - u |}{\|g(u,\cdot)\|_1} \wedge 1 \right); \: h \|g(u,\cdot)\|_1 \ge \sigma_1 \right] \Bigg| \\
    &\le \left| \frac{c}{h \|g(u,\cdot)\|_1}\, \Ee_{(X,N)}^u\left[ \sup_{0 \le r \le h}|\hat{U}_r - u| \wedge \|g(u,\cdot)\|_1;\: h \|g(u,\cdot)\|_1 \ge \sigma_1 \right] \right| \\
    &= \left| \frac{c}{h \|g(u,\cdot)\|_1}\, \Ee_{(X,N)}^u\left[ \sup_{0 \le r \le h} \left|\left( X_r, \int_0^r W(\hat{Y}_t) \, dt\right)\right| \wedge \|g(u,\cdot)\|_1;\: h \|g(u,\cdot)\|_1 \ge \sigma_1 \right] \right| \\
    &\le \left| \frac{c}{h \|g(u,\cdot)\|_1}\, \Ee_{(X,N)}^u\left[ \sup_{0 \le r \le h} \left|\left( X_r, h \|W\|_{\infty}\right)\right| \wedge \|g(u,\cdot)\|_1 ; h \|g(u,\cdot)\|_1 \ge \sigma_1 \right] \right|.
\end{align*}
\end{allowdisplaybreaks}
Using the independence of  $X$  and $\sigma_1$ this is equal to
\begin{align*}
    \frac{c}{h \|g(u,\cdot)\|_1} & \left(1 - e^{- h \|g(u,\cdot)\|_1}\right)
     \Ee_{X}^u  \left[ \sup_{0 \le r \le h} \left|\left( X_r, h \|W\|_{\infty}\right)\right| \wedge \|g(u,\cdot)\|_1  \right] \\
    &\le c\,  \Ee_{X}^u  \left[ \sup_{0 \le r \le h} \left|\left( X_r, h \|W\|_{\infty}\right)\right| \wedge \|g(u,\cdot)\|_1  \right]\\
    &\xrightarrow[h\to 0^+]{} 0
\end{align*}
uniformly for all  $u = (y,s) \in \R \times [-M_1,M_1]$.

It will suffice to show that $\textsf{B} \xrightarrow[h\to 0^+]{} Rf$. Because of the independence of $\eta_1$ and $\sigma_1$ we get
\begin{align}\notag
    \textsf{B}
    &= \Ee_\unif\left[ f(y + \eta_1(u),s) - f(y,s)\right] \, \frac{1}{h} \left(1 - e^{- h \|g(u,\cdot)\|_1}\right)\\\notag
    &= \int_{-\pi}^{\pi} (f(y+x,s) - f(y,s)) ((W(y) - W(y + x))s)_+ \, dx \, \frac{1 - e^{- h \|g(u,\cdot)\|_1}}{h \|g(u,\cdot)\|_1} \\ \notag
    &= R f(y,s) \frac{1 - e^{- h \|g(u,\cdot)\|_1}}{h \|g(u,\cdot)\|_1} \\
    \label{lineR}
    &= R f(y,s)  + R f(y,s) \left( \frac{1 - e^{- h \|g(u,\cdot)\|_1}}{h \|g(u,\cdot)\|_1} - 1 \right).
\end{align}
For $u = (y,s) \in \R \times [-M_1,M_1]$ we have
\begin{gather*}
    |R f(y,s)|
        \le 2 \|f\|_{\infty}\, 2 \pi\, 2\, \|W\|_{\infty} \, M_1
        = 8 \pi\, \|f\|_{\infty} \|W\|_{\infty}\, M_1,\\[\medskipamount]
    \|g(u,\cdot)\|_1
        \le 2 \pi\, 2\, \|W\|_{\infty}\, M_1
        = 4 \pi \,\|W\|_{\infty}\, M_1.
\end{gather*}
Note that for any $h, c > 0$ we have
\begin{equation*}
    -\frac{h c}{2} \le \frac{1 - e^{- hc } - hc}{h c} \le 0.
\end{equation*}
Therefore,
$$
    \left| \frac{1 - e^{- h \|g(u,\cdot)\|_1}}{h \|g(u,\cdot)\|_1} - 1 \right|
    \le
    \frac{h\, \|g(u,\cdot)\|_1}{2} \le \frac{4 \pi \,\|W\|_{\infty}\, M_1}{2} h.
$$
It follows that the expression in \eqref{lineR} tends to $R f(y,s)$ when ${h \to 0^+}$ uniformly for all  $u = (y,s) \in \R \times [-M_1,M_1]$.
We have shown that $\textsf{B} \xrightarrow[ h\to0^+ ]{} Rf$. This was the last step in the proof.
\end{proof}

We will now introduce some  further  notation. Let $\N$ be  the  positive integers and  $\N_0=\N\cup\{0\}$. For any $f: \uni \to \R$ we set
$$
    \tilde{f}(x) := f(e^{ix}), \quad x \in \R.
$$
We say that $f: \uni \to \R$ is \emph{differentiable} at $z= e^{ix}$, $x \in \R$, if and only if $\tilde{f}$ is differentiable at $x$ and we put
$$
    f'(z) := (\tilde{f})'(x), \quad \text{where} \quad  z = e^{ix}, \quad x \in \R.
$$
Analogously, we say that $f: \uni \to \R$ is \emph{$n$ times differentiable}  at $z= e^{ix}$, $x \in \R$, if and only if $\tilde{f}$ is $n$ times differentiable at $x$ and we write
$$
    f^{(n)}(z) = (\tilde{f})^{(n)}(x), \quad \text{where} \quad  z = e^{ix}, \quad x \in \R.
$$
In a similar way we define for $f: \uni \times \R \to \R$
\begin{equation}\label{ftilde2}
    \tilde{f}(y,s) = f(e^{iy},s), \quad y,s \in \R.
\end{equation}
We say that $D^\alpha f(z,s)$, $z = e^{iy}$, $y,s \in \R$, $\alpha\in\N_0^2$, exists if and only if $D^\alpha\tilde{f}(y,s)$ exists and we set
$$
    D^\alpha f(z,s) = D^\alpha\tilde{f}(y,s), \quad \text{where} \quad  z = e^{iy}, \quad  y,s \in \R.
$$
When writing $C^2(\uni)$, $C_c^2(\uni \times \R)$, etc., we are referring to the derivatives defined above.

\begin{proof}[Proof of Corollary \ref{genZS1}]
We will use the notation $\tilde{f}$ introduced in \eqref{ftilde2}. Let $f \in C_c(\uni \times \R)$. Then $\tilde{f} \in C_*(\Rt)$. Let $z = e^{iy}$, $z \in \uni$, $s \in \R$.  We have, cf.\ \cite[eq.\ (2.9)]{BKS},
\begin{align}\label{quotient1}
     \frac{T_t^\uni f(z,s) - f(z,s)}{t}
    &=  \frac{T_t \tilde{f}(y,s) - \tilde{f}(y,s)}{t},\\
\label{quotient2}
     \frac{\hat{T}_t^\uni f(z,s) - f(z,s)}{t}
    &=  \frac{\hat{T}_t \tilde{f}(y,s) - \tilde{f}(y,s)}{t}.
\end{align}
Using this and Proposition \ref{genYS1} we get that $\lim_{t \to 0^+} (T_t^\uni f - f)/t$ exists if and only if $\lim_{t \to 0^+} (\hat{T}_t^\uni f - f)/t$ exists, where both limits are in $||\cdot||_{\infty}$ norm.  Consequently,
$$
    f \in \calD(\calG) \cap C_c(\uni \times \R)
    \iff
    f \in \calD(\hat{\calG}) \cap C_c(\uni \times \R).
$$

The second assertion of the proposition follows
 from \eqref{new-eq}, the definition of the infinitesimal generator and from the fact that for  $z \in \uni$ and $s \in \R$
\begin{equation}\label{RRD}
    R \tilde{f}(y,s) = R^\uni f(z,s), \quad z = e^{iy}.
\qedhere
\end{equation}
\end{proof}

\begin{proof}[Proof of Proposition \ref{genYS2}]
Note that \eqref{formulaYS2a} follows from \eqref{formulaYS2b} by Proposition \ref{genYS1}. So it is sufficient to show \eqref{formulaYS2b}.

Pick  $f \in C_*^2(\Rt)$.  Throughout this proof we assume that $\supp(f) \subset \R \times (-M_0,M_0)$ for some $M_0 > 0$. With exactly the same argument as at the beginning of the proof of Proposition \ref{genYS1}, we can restrict our attention to $(y,s) \in \R \times [-M_1,M_1]$ where $M_1 := M_0 + \|W\|_\infty$. We have for $0 < h < 1$,
\begin{align*}
  &\frac{\hat T_h
 f(y,s) - f(y,s)}{h}\\
    &\qquad = \frac{\Ee^{(y,s)}f(\hat{Y}_h,\hat{S}_h) - \Ee^{(y,s)}f(\hat{Y}_h,s)}{h} + \frac{\Ee^{(y,s)}f(\hat{Y}_h,s) - \Ee^{(y,s)}f(y,s)}{h}\\
    &\qquad = \text{I} + \text{II}.
\end{align*}
We get
\begin{align*}
    \text{I}
    &= \frac{1}{h} \Ee^{(y,s)}\left[ \frac{\partial f}{\partial s} (\hat{Y}_h,\xi) (\hat{S}_h - s)\right]\\
    &= \frac{1}{h} \Ee^{(y,s)} \left[\frac{\partial f}{\partial s} (\hat{Y}_h,\xi)  \int_0^h W(\hat{Y}_t) \, dt\right] \\
    &= \Ee^{(y,s)}\left[ \frac{1}{h} \frac{\partial f}{\partial s} (y,s)  \int_0^h W(y) \, dt\right]
    + \Ee^{(y,s)} \left[\frac{1}{h} \left[\frac{\partial f}{\partial s} (\hat{Y}_h,\xi) - \frac{\partial f}{\partial s} (y,s) \right] \int_0^h W(y) \, dt\right] \\
    &\qquad+   \Ee^{(y,s)} \left[\frac{1}{h} \frac{\partial f}{\partial s} (\hat{Y}_h,\xi) \int_0^h
\left( W(\hat{Y}_t) - W(y) \right)
 \, dt\right]\\
    &= \text{I}_1 + \text{I}_2 + \text{I}_3,
\end{align*}
where $\xi$ is a point between $s$ and $\hat{S}_h$. Note that $|\hat{Y}_h - y| = |X_h|$ and $|\xi - s| \le |\hat{S}_h -s| \le h \|W\|_{\infty}$. Moreover,
\begin{align*}
    |W(\hat{Y}_h) - W(y)|
    &\le \big(2\, \|W\|_{\infty}\big) \wedge \big(\|W'\|_{\infty} |\hat{Y}_h -y|\big)\\
    &\le c\, (|X_h| \wedge 1)\\
    &\le c\, \left(\sup_{0 \le t \le h}|X_t| \wedge 1 \right)
\end{align*}
and
\begin{align*}
    \bigg|\frac{\partial f}{\partial s} & (\hat{Y}_h,\xi) - \frac{\partial f}{\partial s} (y,s) \bigg|\\
    &\le 2\, \left\|\frac{\partial f}{ \partial s}\right\|_{\infty} \wedge \left[ \left(\left\|\frac{\partial^2 f}{ \partial s^2}\right\|_{\infty} + \left\|\frac{\partial^2 f}{ \partial s\, \partial y}\right\|_{\infty}\right) (|\hat{Y}_h - y| + |\xi - s|) \right] \\
    &\le c \,((|X_h| + h) \wedge 1),
\end{align*}
where $c = c(W,f)$. It follows that
$$
    |\text{I}_2| \le c \,\|W\|_{\infty} \, \Ee^{(y,s)}\big((|X_h| + h) \wedge 1\big) \xrightarrow[h\to 0^+]{} 0,
$$
uniformly for all  $u = (y,s) \in \R \times [-M_1,M_1]$. In a similar way
$$
    |\text{I}_3| \le c \left\|\frac{\partial f}{ \partial s}\right\|_{\infty} \, \Ee^{(y,s)}\left[\sup_{0 \le t \le h}|X_t| \wedge 1 \right]
    \xrightarrow[h\to 0^+]{} 0,
$$
uniformly for all  $u = (y,s) \in \R \times [-M_1,M_1]$. So
$$
    \text{I} \xrightarrow[h\to 0^+]{} \frac{\partial f}{\partial s}(y,s) W(y)
$$
uniformly for all  $u = (y,s) \in \R \times [-M_1,M_1]$.

It
 is well known that
$$
    \text{II}
    = \frac{\Ee^{(y,s)}(f(y + X_h,s) - f(y,s))}{h} \xrightarrow[h\to 0^+]{}
    -(-\Delta_y)^{\alpha/2} f(y,s)
$$
uniformly in $u = (y,s)$.

Combining  the  estimates for I and II shows that $f \in \calD(\calG^{(\hat{Y},\hat{S})})$ and that \eqref{formulaYS2b} holds.
\end{proof}

\begin{proof}[Proof of Corollary \ref{genZS2}]
Let $f \in C_c^2(D \times \R)$. Then $\tilde{f} \in C_*^2(\R^2) \subset \calD(\calG^{(Y,S)})$. By (\ref{quotient1}) $f \in \calD(\calG)$. Now let $z = e^{iy}$, $z \in D$, $s \in \R$. By (\ref{quotient1}), Proposition \ref{genYS2}, (\ref{Lz3}) and (\ref{RRD}) we get
\begin{eqnarray*}
\calG f(z,s) &=& \calG^{(Y,S)} \tilde{f}(y,s) \\
&=& -(-\Delta_y)^{\alpha/2} \tilde{f}(y,s) + R\tilde{f}(y,s) + W(y) \tilde{f}_s(y,s) \\
&=& L_z f(z,s) + R^Df(z,s) + V(z) f_s(z,s).
\end{eqnarray*}

The proof for $\hat{\calG}$ is the same.
\end{proof}

\section{Stationary measure}\label{sec3}
The aim of this section is to show that the process $(Z_t,S_t)$ has a unique stationary measure.
First we will show that $C_c^2(\uni \times \R)$ is a core for $(\calG, \calD(\calG))$. For this we will need two auxiliary lemmas.

\begin{lemma} \label{core1}
    $C_c^2(\uni \times \R)$ is a core for $\hat{\calG}$.
\end{lemma}
\begin{proof}
Here we will use the results from \cite{BKS}. Note that $(\hat{Y}_t,\hat{S}_t)$ is the solution of a SDE of the form  (3.1) in \cite{BKS}. Since $V:\uni \to \R$ is a $C^3$ function,  \cite[Theorem 3.1]{BKS}, see also \cite[Proposition 3.6]{BKS}, guarantees that $\hat{T}_tf \in C_*^2(\Rt)$ for all $f \in C_*^2(\Rt)$.

Now let $f\in C_c^2(\uni \times \R)$. Then $\tilde{f} \in C_*^2(\Rt)$ and $\hat{T}_t \tilde{f} \in C_*^2(\Rt)$. For $z = e^{iy}$, $z \in \uni$, $s \in \R$ we get as in \cite[eq.\ (2.9)]{BKS}  $\hat{T}_t^\uni f(z,s) = \hat{T}_t \tilde{f}(y,s)$. Hence, $\hat{T}_t^\uni f \in C_c^2(\uni \times \R)$. This means that $\hat{T}_t^\uni: C_c^2(\uni \times \R) \to C_c^2(\uni \times \R)$. Since $C_c^2(\uni \times \R)$ is dense in $C_0(\uni \times \R)$---the Banach space where the semigroup $\{\hat{T}_t^\uni\}_{t \ge 0}$ is defined---\cite[Proposition 1.3.3]{EK} applies and shows that $C_c^2(\uni \times \R)$ is a core for $(\hat{\calG}, \calD(\hat\calG))$.
\end{proof}

\begin{lemma}\label{core2}
    $C_c(\uni \times \R) \cap \calD(\calG) = C_c(\uni \times \R) \cap \calD(\hat{\calG})$ is a core for $(\calG,\calD(\calG))$ and $(\hat{\calG},\calD(\hat{\calG}))$.
\end{lemma}
\begin{proof}
The equality
of the two families of functions
follows from Corollary \ref{genZS1}.

By Corollary \ref{genZS2}, $C_c^2(\uni \times \R) \subset C_c(\uni \times \R) \cap \calD(\calG)$ and $C_c^2(\uni \times \R)$ is dense in $C_0(\uni \times \R)$ where the semigroups $\{T_t^\uni\}_{t \ge 0}$, $\{\hat{T}_t^\uni\}_{t \ge 0}$ are defined; so $C_c(\uni \times \R) \cap \calD(\calG)$ is dense in $C_0(\uni \times \R)$.

By the definition of the processes $S_t$ and $\hat{S}_t$ and the boundedness of $W$ it is easy to see that $T_t^\uni: C_c(\uni \times \R) \to C_c(\uni \times \R)$ and $\hat{T}_t^\uni: C_c(\uni \times \R) \to C_c(\uni \times \R)$. It follows that $T_t^\uni$ and $\hat{T}_t^\uni$ map $C_c(\uni \times \R) \cap \calD(\calG)$ into itself. Now \cite[Proposition 1.3.3]{EK} gives that $C_c(\uni \times \R) \cap \calD(\calG)$ is a core for $\calG$ and $\hat{\calG}$.
\end{proof}

\begin{proposition}
\label{core3}
    $C_c^2(\uni \times \R)$ is a core for $\calG$.
\end{proposition}
\begin{proof}
Pick $f \in \calD(\calG) \cap C_c(\uni \times \R)$. We have $f \in \calD(\hat{\calG}) \cap C_c(\uni \times \R)$ and $C_c^2(\uni \times \R)$ is a core for $\hat{\calG}$ so there exists a sequence $(f_n)_{n = 1}^{\infty}$, where $f_n \in C_c^2(\uni \times \R)$ such that
$$
    \lim_{n \to \infty} \left(\|f_n - f\|_{\infty} + \|\hat{\calG} f_n - \hat{\calG} f\|_{\infty} \right) = 0.
$$
Since $f \in C_c(\uni \times \R)$, there exists some $M > 0$ such that $\supp(f) \subset \uni \times [-M,M]$. Let $g \in C_c^{\infty}(R)$ be such that $0 \le g \le 1$, $\|g'\|_{\infty} \le 1$, $g \equiv 1$ on $[-M-1,M+1]$ and $g \equiv 0$ on $(-\infty,-M-3] \cup [M+3,\infty)$. Put
$$
    g_n(z,s) := g(s) f_n(z,s), \quad (z,s) \in \uni \times \R,
$$
and note that  $f(z,s)= g(s) f(z,s)$. Therefore
$$
    |g_n(z,s) - f(z,s)|
    = |g(s)f_n(z,s) - g(s)f(z,s)|
    \leq |f_n(z,s) - f(z,s)|
$$
and
$$
    \|g_n - f\|_{\infty} \le \|f_n - f\|_{\infty}.
$$
Since $g_n \in C_c^2(\uni \times \R) \subset \calD(\hat{\calG})$, we find for $(z,s) \in \uni \times [-M,M]$,
\begin{align*}
    |\hat{\calG} g_n(z,s) - \hat{\calG} f(z,s)|
    &= \left|V(z) \frac{\partial g_n}{\partial s}(z,s) + L_z g_n(z,s) - \hat{\calG}f(z,s)\right| \\
    &= \left|V(z) \frac{\partial f_n}{\partial s}(z,s) + L_z f_n(z,s) - \hat{\calG}f(z,s)\right| \\
    &=  |\hat{\calG} f_n(z,s) - \hat{\calG}f(z,s)|\\
    &\le \|\hat{\calG} f_n - \hat{\calG} f\|_{\infty},
\end{align*}
whereas for $(z,s) \notin \uni \times [-M,M]$,
$$
    \hat{\calG} f(z,s) = V(z) \frac{\partial f}{\partial s}(z,s) + L_z f(z,s) = 0,
$$
and
\begin{align*}
    |\hat{\calG} & g_n(z,s) - \hat{\calG} f(z,s)|\\
    &= \left|V(z) \frac{\partial g_n}{\partial s}(z,s) + L_z g_n(z,s) \right| \\
    &= \left|V(z) g'(s) f_n(z,s) + V(z) \frac{\partial f_n}{\partial s}(z,s) g(s)+ g(s) L_z f_n(z,s) \right| \\
    &\le |V(z)|\cdot |g'(s)|\cdot |f_n(z,s)| + |g(s)| \left|V(z)\, \frac{\partial f_n}{\partial s}(z,s) +  L_z f_n(z,s)\right| \\
    &\le \|V\|_{\infty} |f_n(z,s)| + |\hat{\calG} f_n(z,s)| \\
    &= \|V\|_{\infty} |f_n(z,s) - f(z,s)| + |\hat{\calG} f_n(z,s) - \hat{\calG} f(z,s)| \\
    &\le \|V\|_{\infty} \|f_n - f\|_{\infty}+ \|\hat{\calG} f_n - \hat{\calG} f\|_{\infty}.
\end{align*}
Hence
$$
    \|g_n - f\|_{\infty} + \|\hat{\calG} g_n - \hat{\calG} f\|_{\infty}
    \le
    (1 + \|V\|_{\infty}) \|f_n - f\|_{\infty} +\|\hat{\calG} f_n - \hat{\calG} f\|_{\infty}
$$
and we see that
$$
    \lim_{n \to \infty} \left(\|g_n - f\|_{\infty} + \|\hat{\calG} g_n - \hat{\calG} f\|_{\infty}\right) = 0.
$$

Note that for every $M > 0$ there exists a constant $C_{M,V} > 0$ such that
$$
    \|R^\uni h\|_{\infty} \le C_{M,V} \|h\|_{\infty},
$$
for all $h \in C_c(\uni \times \R)$ such that $\supp(h) \subset \uni \times [-M + 3, M + 3]$. Hence
\begin{align*}
    \|g_n - f &\|_{\infty} + \|\calG g_n - \calG f\|_{\infty}\\
    &= \|g_n - f\|_{\infty} + \|\hat{\calG} g_n - \hat{\calG} f + R^\uni g_n - R^\uni f\|_{\infty}\\
    &\le \|g_n - f\|_{\infty} +\|\hat{\calG} g_n - \hat{\calG} f \|_{\infty} + \|R^\uni g_n - R^\uni f\|_{\infty}\\
    &\le (1 + \|V\|_{\infty})\|f_n - f\|_{\infty} + \|\hat{\calG} f_n - \hat{\calG} f \|_{\infty} + C_{M,V}\|g_n - f\|_{\infty}\\
    &\le (1 + \|V\|_{\infty} + C_{M,V})\|f_n - f\|_{\infty}+ \|\hat{\calG} f_n - \hat{\calG} f \|_{\infty}\\
    &\xrightarrow[n\to\infty]{} 0.
\end{align*}
This shows that for every $f \in \calD(\calG) \cap C_c(\uni \times \R)$ there exists a sequence  $(g_n)_{n = 1}^{\infty}$, such that $g_n \in C_c^2(\uni \times \R)$ and
$$
    \|g_n - f\|_{\infty} + \|\calG g_n - \calG f\|_{\infty} \xrightarrow[n\to\infty]{} 0.
$$
Since we know that $\calD(\calG) \cap C_c(\uni \times \R)$ is a core for $(\calG,\calD(\calG))$, we conclude that $C_c^2(\uni \times \R)$ is also a core for $(\calG,\calD(\calG))$.
\end{proof}

We will now indentify the form of the stationary distribution of the process $(Z_t,S_t)$. For this we need two auxiliary results, Lemma \ref{Lemma2.8BKS} and Proposition \ref{stationary}. Since Lemma \ref{Lemma2.8BKS} is crucial for our argument we reproduce its short proof from \cite[Lemma 2.8]{BKS}.
\begin{lemma}\label{Lemma2.8BKS}
For any  $f \in C^2(\uni)$ we have
$$
    \int_\uni Lf(z) \, dz = 0.
$$
\end{lemma}

\begin{proof}
Recall that $\Arg(z)$ denotes the argument of $z \in \C$ belonging to $(-\pi,\pi]$. First we will show that
\begin{equation}\label{intL}
    \iint_{\uni\times \uni} \I_{\{w\::\:|\Arg(w/z)| > \eps\}}(w) \, \frac{f(w) - f(z)}{|\Arg(w/z)|^{1 + \alpha}} \, dw \, dz = 0.
\end{equation}
We interchange $z$ and $w$, use Fubini's theorem and observe that $|\Arg(z/w)| = |\Arg(w/z)|$,
\begin{align*}
    \iint_{\uni\times \uni} &\I_{\{w\::\:|\Arg(w/z)| > \eps\}}(w) \, \frac{f(w) - f(z)}{|\Arg(w/z)|^{1 + \alpha}} \, dw \, dz\\
    &= \iint_{\uni\times \uni}  \I_{\{z\::\:|\Arg(z/w)| > \eps\}}(z)  \, \frac{f(z) - f(w)}{|\Arg(z/w)|^{1 + \alpha}} \, dz \, dw\\
    &= \iint_{\uni\times \uni}  \I_{\{z\::\:|\Arg(z/w)| > \eps\}}(z)  \, \frac{f(z) - f(w)}{|\Arg(z/w)|^{1 + \alpha}} \, dw \, dz\\
    &= -\iint_{\uni\times \uni} \I_{\{w\::\:|\Arg(w/z)| > \eps\}}(w) \, \frac{f(w) - f(z)}{|\Arg(w/z)|^{1 + \alpha}} \, dw \, dz,
\end{align*}
which proves \eqref{intL}.

By interchanging $z$ and $w$ we also get that
\begin{equation}\label{intL2}\begin{aligned}
    \sum_{n \in \Z \setminus \{0\}} &\int_\uni \int_{\uni} \frac{f(w) - f(z)}{|\Arg(w/z) + 2 n \pi|^{1 + \alpha}} \, dw \, dz \\
    &= \sum_{n \in \Z \setminus \{0\}} \int_\uni \int_{\uni} \frac{f(z) - f(w)}{|\Arg(z/w) + 2 n \pi|^{1 + \alpha}} \, dz \, dw.
\end{aligned}\end{equation}
Note that for $\Arg(w/z) \ne \pi$ we have $|\Arg(z/w) + 2 n \pi| = |\Arg(w/z) - 2 n \pi|$. Hence the expression in \eqref{intL2} equals $0$.

Set
$$
    L_{\eps}f(z) := \int_{\uni \cap \{|\Arg(w/z)| > \eps\}} \frac{f(w) - f(z)}{|\Arg(w/z)|^{1 + \alpha}} \, dw.
$$
What is left is to show that
\begin{equation}\label{intLeps}
    \int_\uni \lim_{\eps \to 0^+} L_{\eps}f(z) \, dz = \lim_{\eps \to 0^+} \int_\uni L_{\eps}f(z) \, dz.
\end{equation}
By the Taylor expansion we have for $f \in C^2(\uni)$
$$
    f(w) - f(z) = \Arg(w/z) f'(z) + \Arg^2(w/z) r(w,z), \quad w, \, z \in \uni,
$$
where $|r(w,z)| \le c(f)$.  Hence,
\begin{align*}
    |L_{\eps}f(z)|
    &= \left|\int_{\uni \cap \{|\Arg(w/z)| > \eps\}} r(w,z) \Arg^{1 - \alpha}(w/z) \, dw\right| \\
    &\le c(f) \int_\uni |\Arg^{1 - \alpha}(w/z)| \, dw = c(f,\alpha).
\end{align*}
Therefore, we get \eqref{intLeps} by the bounded convergence theorem.
\end{proof}

\begin{proposition}\label{stationary}
    Let
    $$
        \pi(dz,ds) = \frac{1}{2 \pi} e^{-\pi s^2} \, dz \, ds.
    $$
    Then for any $f \in C_c^2(\uni \times \R)$ we have
    $$
        \int_\uni \int_{\R} \calG f(z,s) \, \pi(dz,ds) = 0.
    $$
\end{proposition}
\begin{proof}
Let $f \in C_c^2(\uni \times \R)$. By Corollary \ref{genZS2} we have
\begin{align*}
    &2 \pi \int_\uni \int_{\R} \calG f(z,s) \, \pi(dz,ds) \\
    &= \int_{\R} \int_\uni L_z f(z,s) \, dz\, e^{-\pi s^2} \, ds
        + \int_\uni V(z) \int_{\R} f_s(z,s) e^{-\pi s^2} \, ds \, dz \\
    &\qquad\mbox{} + \int_{\R} \int_\uni \int_\uni (f( w,s) - f(z,s)) ((V(z) - V( w))s)_{+} \, d w \, dz \, e^{-\pi s^2} \, ds \\
    &= \text{I} + \text{II} + \text{III}.
\end{align*}
 From  Lemma  \ref{Lemma2.8BKS} we know that $\text{I} = 0$.
Integrating by parts we obtain
$$
    \text{II} = 2 \pi \int_\uni \int_{\R} V(z) f(z,s) e^{-\pi s^2} s \, ds \, dz.
$$
Now we will simplify $\text{III}$. Note that $a_+ = (a + |a|)/2$, $a \in \R$. Hence
\begin{align*}
    \int_\uni \int_\uni &(f( w,s) - f(z,s)) ((V(z) - V( w))s)_{+} \, d w \, dz \\
    &= \frac{s}{2} \int_\uni \int_\uni (f( w,s) - f(z,s)) (V(z) - V( w)) \, d w \, dz \\
    &\qquad+ \frac{|s|}{2} \int_\uni \int_\uni (f( w,s) - f(z,s)) |V(z) - V( w)| \, d w \, dz\\
    &= \text{III}_1 + \text{III}_2.
\end{align*}
By interchanging $w$ and $z$ in $\text{III}_2$ we get
$$
    \text{III}_2
    =  \frac{|s|}{2} \int_\uni \int_\uni (f(z,s) - f( w,s)) |V( w) - V(z)| \, d w \, dz
    = - \text{III}_2,
$$
which means that $\text{III}_2 = 0$.

By assumption, $\int_\uni V(z) \, dz = 0$. Therefore
\begin{align*}
    \text{III}_1
    &= \frac{s}{2} \int_\uni f( w,s) \, d w \int_\uni V(z) \, dz - \frac{s}{2} \int_\uni f( w,s) V( w) \, d w \int_\uni \, dz \\
    &\qquad- \frac{s}{2} \int_\uni f(z,s) V(z) \, dz \int_\uni \, d w + \frac{s}{2} \int_\uni f(z,s) \, dz \int_\uni V( w) \, d w \\
    &= - 2 \pi s \int_\uni f(z,s) V(z) \, dz.
\end{align*}
Informally, $\text{III} = \int \left((\text{III}_1) e^{-\pi s^2}\right)ds$, so
$$
    \text{III} = - 2 \pi \int_\uni \int_{\R} V(z) f(z,s) e^{-\pi s^2} s \, ds \, dz.
$$
Consequently $\text{I} + \text{II} + \text{III} = 0$.
\end{proof}

\begin{theorem}\label{main2}
    The measure
    \begin{align}\label{5.6.1}
        \pi(dz,ds) = \frac{1}{2 \pi} e^{-\pi s^2} \, dz \, ds.
    \end{align}
is a stationary distribution of the process $(Z_t,S_t)$.
\end{theorem}
\begin{proof}
Let $(Y_t,S_t)$ be the Markov process given by \eqref{rep12} and let $(Z_t,S_t)$ be the Markov process where $Z_t = e^{i Y_t}$. By  $\{T_t^\uni\}_{t\ge 0}$ we denote the transition semigroup of $(Z_t,S_t)$ on the Banach space $C_0(\uni \times \R)$, cf.\ \eqref{TtD1}, and by $\calG$ we denote its generator.
Let $\calP(\R \times \R)$ and $\calP(\uni \times \R)$ denote the sets of all probability measures on $\R \times \R$ and $\uni  \times \R$ respectively. In this proof, for any $\tilde\mu \in \calP(\uni \times \R)$ we define $\mu \in \calP(\R \times \R)$ by $\mu([0,2\pi)\times \R) = 1$ and $\mu(A \times B) = \tilde\mu(e^{iA} \times B)$ for Borel sets $A \subset [0,2\pi)$, $B \subset \R$.

Consider any $\tilde\mu \in \calP(\uni \times \R)$ and the corresponding $\mu \in \calP(\R \times \R)$. For this $\mu$ there exists a Markov process $(Y_t,S_t)$ given by \eqref{rep12} such that $(Y_0,S_0)$ has the distribution $\mu$. It follows that for any $\wt\mu \in \calP(\uni \times \R)$ there exists a Markov process $(Z_t,S_t)$ with $Z_t = e^{i Y_t}$ and with initial distribution $\wt\mu$. By \cite[Proposition 4.1.7]{EK} $(Z_t,S_t)$ is a solution of the martingale problem for $(\calG,\wt\mu)$. By \cite[Theorem 4.4.1]{EK}
for any $\wt\mu \in \calP(\uni \times \R)$, uniqueness holds for the martingale problem for $(\calG,\calD(\calG),\wt\mu)$.  Hence the martingale problem for $\calG$ is well posed.

Proposition \ref{core3} gives that $C_c^2(\uni \times \R)$ is a core for $\calG$. By Proposition \ref{stationary} and \cite[Proposition 4.9.2]{EK} we get that $\pi$ is a stationary measure for $\calG$. This means that $\pi$ is a stationary distribution for $(Z_t,S_t)$.
\end{proof}

\begin{theorem}\label{uniq2}
The measure $\pi$ defined in \eqref{5.6.1} is the unique stationary distribution of the process $(Z_t,S_t)$.
\end{theorem}
\begin{proof}
The proof is similar to the proof of \cite[Theorem 2.12]{BKS}.

\noindent \emph{Step 1}. Suppose that $(Y_t,S_t)$ satisfies
\begin{gather*}
    Y_t =  y + X_t , \qquad
    S_t = s + \int_0^t W(Y_r) \, dr, 
\end{gather*}
where $X_0=0$. Suppose that $X_t$ is a stable L\'evy process with $X_0=0$. The following   L\'evy inequality for symmetric L\'evy processes  is well known
\begin{align*}
    \Pp\left(\sup_{0\leq r\leq\tau} |X_r|>\epsilon\right) \leq 2\,\Pp\left( |X_\tau| > \epsilon\right) \leq 1-\delta.
\end{align*}
It follows that for every $\tau< \infty$,  $y,s\in \R$ and  $\eps>0$ there exists $\delta>0$ such that
\begin{align}\label{5.7.1}
     \Pp^{y,s}\left(\sup_{0\leq r \leq \tau}|Y_r - y| \leq \eps \right)
    =
    \Pp\left(\sup_{0\leq r \leq \tau}|X_r | \leq \eps \right) \geq \delta.
\end{align}

\medskip \noindent \emph{Step 2}.
Recall that $V\in C^3$ and it is not identically constant. This and the fact that $\int_\uni V (z) \,dz =0$ imply that $W$ is strictly positive on some interval and  strictly negative on some other interval. We fix some $a_1, a_2\in(-\pi,\pi)$, $b_1>0$, $b_2<0$ and $\eps_0\in(0,\pi/100)$, such that $V(z)>b_1$ for $z\in \uni$, $\Arg(z) \in [a_1-4\eps_0, a_1 + 4\eps_0]$, and $V(z)<b_2$ for $z\in \uni$, $\Arg(z) \in [a_2-4\eps_0, a_2 + 4\eps_0]$.

Suppose that there exist two stationary probability distributions $\pi$ and $\wh\pi$ for $(Z,S)$. Let $((Z_t,
S_t))_{t\geq0}$ and $((\wh Z_t, \wh S_t))_{t\geq0}$ be processes with $(Z_0,S_0)$ and $(\wh Z_0, \wh S_0)$ distributed according to $\pi$ and $\wh\pi$, respectively. The transition probabilities for these processes are the same as for the processes which are solutions to  \eqref{rep12}. Recall that $X$ denotes the driving stable L\'evy process for $Z$ and $\tau_1$ is the time of the first ``extra jump'' in the representation \eqref{rep12}.

We will show that $S_t\ne0$ for some $t>0$, a.s. Suppose that the event $A = \{S_t = 0 \text{ for all } t\geq 0\}$ has strictly positive probability.  On $A$ we have $Y_t = X_t+y$ for all $t\geq 0$, according to \eqref{rep12}. Recall that $W(x) > 0$ for all $x$ in the set $\Gamma := \bigcup_{k \in \Z} (a_1-4\eps_0+2\pi k, a_1 + 4\eps_0+ 2 k \pi)$. It is easy to see that $X$ enters  $\Gamma-y$  at a finite time $s_0$, a.s. Hence, $Y$ enters $\Gamma$ at a finite time $s_0$, on the event $A$. Since $Y$ is right-continuous, $Y_t \in \Gamma$ for all $t\in (s_0,s_1)$ for some random $s_1>s_0$. This and \eqref{rep12} imply  that $S_t \ne 0$ for some $t\in (s_0,s_1)$, on the event $A$. This contradicts the definition of $A$ and hence it proves that $S_t\ne0$ for some $t>0$, a.s.

Assume without loss of generality that $S_t>0$ for some $t>0$, with positive probability. Then there exist $\eps_1>0$, $t_1>0$ and $p_1>0$ such that
$$
    \Pp^\pi(S_{t_1}>\eps_1, \tau_1>t_1) > p_1.
$$
Let $F_1 = \{S_{t_1}>\eps_1, \tau_1>t_1\}$ and $t_2 = \eps_1/(2\|W\|_\infty)$. Clearly, for some $p_2>0$ we have
$$
    \Pp^{\pi}\left(\exists\,t\in [t_1, t_1+t_2]\::\: \Arg(Z_t) \in [a_2-\eps_0, a_2 + \eps_0],\: \tau_1 > t_1+t_2 \:\big|\: F_1\right) > p_2.
$$
 Since $\Arg(Z_t)$ has right-continuous paths,  this implies that there exist $\eps_1>0$, $t_1>0$, $t_3 \in [t_1, t_1+t_2]$ and $p_3>0$ such that
\begin{align*}
    \Pp^\pi(S_{t_1}>\eps_1, \; \Arg(Z_{t_3}) \in [a_2-2\eps_0, a_2 + 2\eps_0], \tau_1>t_3) > p_3.
\end{align*}
Note that $|S_{t_3} - S_{t_1}| \leq \|W\|_\infty\, t_2 < \eps_1/2$. Hence,
\begin{align*}
    \Pp^\pi(S_{t_3}>\eps_1/2, \; \Arg(Z_{t_3}) \in [a_2-2\eps_0, a_2 + 2\eps_0], \tau_1>t_3) > p_3.
\end{align*}
Let $\eps_2 \in( \eps_1/2,\infty)$ be such that
\begin{align*}
    \Pp^\pi(S_{t_3}\in[\eps_1/2,\eps_2], \; \Arg(Z_{t_3}) \in [a_2-2\eps_0, a_2 + 2\eps_0], \tau_1>t_3) > p_3/2.
\end{align*}
Set $t_4 = 2\eps_2 /|b_2|$ and $t_5 = t_3 + t_4$. By \eqref{5.7.1}, for any $\eps_3>0$ and some $p_4>0$,
\begin{align*}
    \Pp^\pi\Big(&\sup_{t_3\leq r \leq t_5}|X_r - X_{t_3}| \leq \eps_3,\; S_{t_3}\in[\eps_1/2,\eps_2],\\
    &\Arg(Z_{t}) \in [a_2-3\eps_0, a_2 + 3\eps_0] \text{\ \  for all \ \ } t\in [t_3,  t_5],\: \tau_1>t_5\Big) > p_4.
\end{align*}
Since $V(x) < b_2 < 0$ for $ x \in [a_2-3\eps_0, a_2 + 3\eps_0]$, if the event in the last formula holds then
\begin{align*}
    S_{t_5} = S_{t_3} + \int _{t_3}^{t_5} V(Z_s) ds
    \leq \eps_2 + b_2 t_4 \leq -\eps_2.
\end{align*}
This implies that,
\begin{gather}\label{5.7.21}
    \Pp^\pi\Big(\sup_{t_3\leq r \leq t_5}|X_r - X_{t_3}| \leq \eps_3,\: S_{t_3} \geq \eps_1/2, S_{t_5} \leq - \eps_2, \tau_1>t_5\Big) > p_4.
\end{gather}

\medskip \noindent \emph{Step 3}. By the L\'evy-It\^o representation we can write the stable L\'evy process $X$ in the form $X_t =  J_t + \wt X_t$, where $ J$ is a compound Poisson process comprising all jumps of $X$ which are greater than $\eps_0$ and $\wt X = X- J$ is an independent L\'evy process (accounting for all small jumps of $X$).  Denote by $\lambda = \lambda(\alpha,\eps_0)$ the rate of the compound Poisson process $J$ and let $(\wt Y, \wt S)$ be the solution to \eqref{rep12}, with $X_t$ replaced by $\wt X_t$ for $t\geq t_3$.
 Similarly $\wt \tau_1$ denotes the first "extra jump" in  the representation \eqref{rep12} for the process $(\wt Y, \wt S)$.  Moreover, we take $\eps_3 < \eps_0/2$. By our construction $\sup_{t_3\leq r \leq t_5}|X_r - X_{t_3}| \le \eps_3$ entails that $\sup_{t_3\leq r \leq t_5}|J_r- J_{t_3}|=0$; therefore, \eqref{5.7.21} becomes
\begin{align*}
    &\Pp^\pi\Big(\sup_{t_3\leq r \leq t_5}|\wt X_r - \wt X_{t_3}| \leq \eps_3,\:  \wt S_{t_3} \geq \tfrac{\eps_1}2,\: \wt S_{t_5} \leq - \eps_2,\:  \wt \tau_1  >t_5\Big) \\
    &\geq \Pp^\pi\Big(\sup_{t_3\leq r \leq t_5}|\wt X_r - \wt X_{t_3}| \leq \eps_3,\: \sup_{t_3\leq r \leq t_5}| J_r - J_{t_3}|=0,\:  \wt S_{t_3} \geq \tfrac{\eps_1}2,\: \wt S_{t_5} \leq - \eps_2,\:  \wt \tau_1  >t_5\Big) \\
    &> p_4>0.
\end{align*}

Let $\tau$ be the time of the first jump of $ J$ in the interval $[t_3, t_5]$; we set $\tau=t_5$ if there is no such jump. We can represent $\{(Y_t,S_t), 0\leq t \leq \tau\}$ in the following way: $(Y_t,S_t) = (\wt Y_t, \wt S_t)$ for  $0\leq t < \tau$, $S_\tau = \wt S_\tau$, and $Y_\tau = \wt Y_{ \tau -} +  J_{\tau} -  J_{\tau-}$.  Note that  $\wt Y_t = y + \wt X_t$  if $t<\tau_1$.

We say that a non-negative measure $\mu_1$ is a component of a non-negative measure $\mu_2$ if $\mu_2 = \mu_1 + \mu_3$ for some non-negative measure $\mu_3$. Let $\mu(dz,ds) = \Pp^\pi(Z_\tau\in dz, S_\tau \in ds)$. We will argue that $\mu(dz, ds)$ has a component with a density bounded below by $c_2 >0$ on $\uni \times (-\eps_2, \eps_1/2)$.  We find  for every Borel set $A\subset\uni$ of arc length $|A|$ and every interval $(s_1,s_2) \subset (-\eps_2, \eps_1/2)$
\begin{align*}\small
    &\mu(A\times (s_1,s_2))\\
    &= \Pp^\pi\left(Z_\tau\in A,\: S_\tau\in (s_1,s_2)\right)\\
    &\geq \Pp^\pi\Big(Z_\tau\in A, S_\tau\in (s_1,s_2),  \sup_{t_3\leq r \leq t_5}|\wt X_r - \wt X_{t_3}| \leq \eps_3,\:  \wt S_{t_3} \geq \tfrac{\eps_1}2,\:  \wt S_{t_5} \leq - \eps_2,\: \wt \tau_1>t_5  \Big)\\
    &\geq \Pp^\pi\Big(e^{i(J_\tau- J_{\tau-})}\in e^{-i\wt X_{\tau-}} A,\: \wt S_\tau\in (s_1,s_2),\\
    &\qquad \sup_{t_3\leq r \leq t_5}|\wt X_r - \wt X_{t_3}| \leq \eps_3,\wt S_{t_3} \geq \eps_1/2,\: \wt S_{t_5} \leq - \eps_2,\:  \wt \tau_1  >t_5,\: N^J=1\Big).
\end{align*}
Here $N^J$ counts the number of jumps of the process $J$ occurring during the interval $[t_3,t_5]$. Without loss of generality we can assume that $\eps_0 < 2\pi$. In this case the density of the jump measure of $J$ is bounded below by $c_3>0$ on $(2\pi,4\pi)$. Observe that  the processes $(\wt X, \wt S)$ and $J$ are independent.  Conditional on $\{N^J=1\}$, $\tau$ is uniformly distributed on $[t_3,t_5]$, and the probability of the event $\{N^J = 1\}$ is $\lambda (t_5 - t_3) e^{-\lambda (t_5 - t_3)}$. Thus,
\begin{align*}
    &\mu(A\times (s_1,s_2))  \geq \\
    & c_3 |A| \Pp^\pi\Big(\wt S_\tau\in (s_1,s_2) \Big|\: \sup_{t_3\leq r \leq t_5}|\wt X_r - \wt X_{t_3}| \leq \eps_3, \wt S_{t_3} \geq \eps_1/2, \wt S_{t_5} \leq - \eps_2,  \wt \tau_1>t_5, N^J = 1\Big)\\
    &\qquad  \times   p_4\cdot \lambda (t_5 - t_3) e^{-\lambda (t_5 - t_3)}.
\end{align*}
Since the process $\wt S$ spends at least  $(s_2-s_1)/\|W\|_\infty$   units of time in $(s_1,s_2)$ we finally arrive at
$$
    \mu(A,  (s_1,s_2) ) \geq  p_4 \lambda e^{-\lambda (t_5 - t_3)}  c_3 |A|  (s_2-s_1)/\|W\|_\infty.
$$
This proves that $\mu(dz, ds)$ has a component with a density bounded below by $c_2=  p_4 \lambda e^{-\lambda (t_5 - t_3)}  c_3/\|W\|_\infty $ on $\uni \times (-\eps_2, \eps_1/2)$.

\medskip\noindent \emph{Step 4}. Let $ \eps_4 = \frac{\eps_1}2 \land \eps_2>0$. We
have shown that for some stopping time $\tau$, $ \Pp^\pi(Z_\tau\in
dz, S_\tau \in ds)$ has a component with a density bounded below by
$c_2>0$ on $\uni \times (-\eps_4, \eps_4)$. We can prove in an
analogous way that for some stopping time $\wh\tau$ and $\wh
\eps_4>0$, $ \Pp^{\wh\pi}(\wh Z_{\wh\tau}\in dz, \wh S_{\wh\tau} \in
ds)$ has a component with a density bounded below by $\wh c_2>0$ on
$\uni \times (-\wh\eps_4, \wh\eps_4)$.

Since $\pi \ne \wh\pi$, there exists a Borel set $A\subset \uni
\times \R$ such that $\pi(A) \ne \wh \pi(A)$. Moreover, since
any two stationary probability measures are either mutually
singular or identical,  cf.\ \cite[Chapter 2, Theorem 4]{S}, we have $\pi(A)>0$ and $\wh\pi(A) =0$
for some $A$. By the strong Markov property applied at $\tau$
and the ergodic theorem, see \cite[Chapter 1, page 12]{S}, we have $\Pp^\pi$-a.s.
\begin{align*}
    \lim_{t\to \infty} (1/t) \int_\tau^t \I_{ \{(Z_s, S_s) \in A\} } \,ds = \pi(A)>0.
\end{align*}
Similarly, we see that $\Pp^{\wh\pi}$-a.s.
\begin{align*}
    \lim_{t\to \infty} (1/t) \int_{\wh\tau}^t \I_{\{(\wh Z_s,\wh S_s) \in A\} } \, ds = \wh\pi(A)=0.
\end{align*}
Since the distributions of $( Z_{\tau},  S_{\tau})$ and $(\wh
Z_{\wh\tau}, \wh S_{\wh\tau})$ have mutually absolutely continuous
components,  the last two statements contradict each other. This
shows that we must have $\pi = \wh\pi$.
\end{proof}

\end{document}